\documentclass{article}
\usepackage{amsmath,amssymb,amsthm,color,graphicx,diagbox,colortbl,array}
\usepackage{tikz}
\usepackage{setspace}
\usepackage{bussproofs}
\usepackage{multicol}
\usepackage{subfiles}

\theoremstyle{plain}
\newtheorem{thm}{Theorem}[section]
\newtheorem*{thm*}{Theorem}
\newtheorem{lem}[thm]{Lemma}
\newtheorem{prop}[thm]{Proposition}
\newtheorem*{prop*}{Proposition}
\newtheorem{cor}[thm]{Corollary}

\newtheorem*{fact*}{Fact}

\newtheorem*{prob*}{Problem}
\newtheorem*{cl*}{Claim}

\newtheorem*{scl*}{Subclaim}

\theoremstyle{definition}
\newtheorem{defn}[thm]{Definition}
\newtheorem{ex}[thm]{Example}
\newtheorem{rem}[thm]{Remark}

\newcommand{\limp}{\to}

\newcommand{\proves}{\vdash}

\newcommand{\forces}{\Vdash}

\newcommand{\PropVar}{\mathrm{Prop}}
\newcommand{\FmlPL}{\mathrm{Fml}_\mathrm{P}}
\newcommand{\FmlML}{\mathrm{Fml}_\mathrm{M}}

\newcommand{\Vis}{\mathcal{V}}
\newcommand{\lconj}{\bigwedge}
\newcommand{\ldisj}{\bigvee}

\newcommand{\Logic}[1]{\mathbf{#1}}
\newcommand{\LogicGL}{\Logic{GL}}
\newcommand{\LogicD}{\Logic{D}}
\newcommand{\LogicS}{\Logic{S}}
\newcommand{\LogicK}{\Logic{K}}
\newcommand{\LogicKFour}{\Logic{K4}}
\newcommand{\LogicSFour}{\Logic{S4}}

\newcommand{\LogicInt}{\Logic{Int}}
\newcommand{\LogicBPL}{\Logic{BPL}}
\newcommand{\LogicFPL}{\Logic{FPL}}
\newcommand{\LogicDPL}{\Logic{DPL}}
\newcommand{\LogicSPL}{\Logic{SPL}}

\newcommand{\Theory}[1]{\mathbf{#1}}
\newcommand{\TheoryPA}{\Theory{PA}}
\newcommand{\TheoryTA}{\Theory{TA}}

\newcommand{\seqI}{\Rightarrow^1}
\newcommand{\seqII}{\Rightarrow^2}
\newcommand{\seqIII}{\Rightarrow^3}
\newcommand{\seq}[1]{\mathrel{\Rightarrow^{#1}}}
\newcommand{\lev}{\ell}
\newcommand{\Inv}[1]{{#1}^{-1}}

\newcommand{\RAx}[1]{(\mathrm{Ax}_{#1})}
\newcommand{\RBotL}[1]{({\bot}\mathrm{L}_{#1})}
\newcommand{\RWL}[2]{(\mathrm{WL}^{#1}_{#2})}
\newcommand{\RWR}[2]{(\mathrm{WR}^{#1}_{#2})}
\newcommand{\RImpL}[2]{({\limp}\mathrm{L}^{#1}_{#2})}
\newcommand{\RImpR}[2]{({\limp}\mathrm{R}^{#1}_{#2})}
\newcommand{\RAndL}[2]{({\land}\mathrm{L}^{#1}_{#2})}
\newcommand{\RAndR}[2]{({\land}\mathrm{R}^{#1}_{#2})}
\newcommand{\ROrL}[2]{({\lor}\mathrm{L}^{#1}_{#2})}
\newcommand{\ROrR}[2]{({\lor}\mathrm{R}^{#1}_{#2})}
\newcommand{\RBoxL}[2]{(\Box\mathrm{L}^{#1}_{#2})}
\newcommand{\RBoxGL}[2]{(\Box\LogicGL^{#1}_{#2})}
\newcommand{\RBoxKFour}[2]{(\Box\LogicKFour^{#1}_{#2})}
\newcommand{\RLift}[2]{(\mathrm{LU}^{#1}_{#2})}
\newcommand{\RCut}[2]{(\mathrm{Cut}^{#1}_{#2})}

\newcommand{\pagecenter}[1]{%
  \par\addvspace{\abovedisplayskip}%
  \noindent\makebox[\textwidth][c]{#1}%
  \par\addvspace{\belowdisplayskip}%
}

\newcommand{\Gentzen}[1]{\mathsf{G}_{#1}}
\newcommand{\GentzenS}{\Gentzen{\LogicS}}
\newcommand{\GentzenD}{\Gentzen{\LogicD}}
\newcommand{\GentzenSPL}{\Gentzen{\LogicSPL}}
\newcommand{\GentzenDPL}{\Gentzen{\LogicDPL}}

\title{Embeddings of Propositional Logics into the Provability Logics $\LogicS$ and $\LogicD$}
\author{Mashu Noguchi\footnote{Email: me@sno2wman.net}
\footnote{Graduate School of System Informatics, Kobe University, 1-1 Rokkodai, Nada, Kobe 657-8501, Japan.}}
\date{}

\begin{document}

\maketitle

\begin{abstract}
  Just as Visser showed that the formal propositional logic $\LogicFPL$ can be embedded into Gödel-Löb provability logic $\LogicGL$,
  Petrukhin proposed a propositional logic $\LogicSPL$ that can be embedded into Solovay's non-normal provability logic $\LogicS$.
  In this paper, we fix Petrukhin's proof and
  extend the result to Japaridze's provability logic $\LogicD$,
  and propose a propositional logic $\LogicDPL$ that can be embedded into $\LogicD$.
\end{abstract}

\section{Introduction}

In 1976, Solovay \cite{Sol76} proved the arithmetical completeness theorem, which is a fundamental landmark result in provability logic.
More precisely, he proved that a modal formula is provable in the logic $\LogicGL$ if and only if
every arithmetical sentence, obtained from the modal formula by interpreting the propositional variables as arithmetical sentences
and the modal operator $\Box$ as an appropriate provability predicate $\mathrm{Pr}_{\TheoryPA}(x)$ of $\TheoryPA$, is provable in $\TheoryPA$.
As an extension of this result, he also showed that a non-normal extension of $\LogicGL$, known as $\LogicS$, is arithmetically complete with respect to
true arithmetic $\TheoryTA$.
Furthermore, Japaridze (also transliterated as Dzhaparidze) discovered a provability logic $\LogicD$
that is strictly intermediate between $\LogicGL$ and $\LogicS$ in his thesis \cite{Jap86}.
The logic $\LogicD$ is the provability logic of $\TheoryPA$ relative to $\TheoryPA + \mathrm{Rfn}_{\TheoryPA}(\Sigma_1)$,
that is, any modal formula is provable in $\LogicD$ if and only if
every arithmetical sentence obtained from the formula by above interpretation is provable in $\TheoryPA + \mathrm{Rfn}_{\TheoryPA}(\Sigma_1)$.
Here, $\mathrm{Rfn}_{\TheoryPA}(\Sigma_1)$ denotes the reflection principle for $\Sigma_1$-sentences.
Subsequently, Beklemishev \cite{Bek87} proved the classification theorem for provability logics,
which states that every provability logic must belong to one of four families of logics,
each being an appropriate extension of one of $\LogicGL$, $\LogicS$, or $\LogicD$.
For more results and detailed discussions about $\LogicD$ and the classification theorem,
see Beklemishev's papers \cite{Bek87,Bek89,Bek90} or the survey by Artemov and Beklemishev \cite{AB05}.


Sequent calculi for provability logics have been well studied, particularly for $\LogicGL$ \cite{Lei81, SV82, Val83, Bor83, Avr84, Moe01, GR12, Bri16, GRS21}.
Kushida \cite{Kus20} proposed a new sequent calculus for $\LogicS$, which has two levels of sequents $\seqI$ and $\seqII$,
and proved cut-elimination for the calculus.
Roughly speaking, sequents with $\seqI$ correspond to provability in $\LogicGL$ and sequents with $\seqII$ correspond to provability in $\LogicS$.
There is a lift-up rule from $\seqI$ to $\seqII$, but not vice versa.
Extending this approach, Kashima et al.\ \cite{KKIM25} proposed a new sequent calculus for $\LogicD$,
which has three levels of sequents $\seqI$, $\seqII$ and $\seqIII$.
Sequents with $\seqI$ and $\seqII$ correspond to provability in $\LogicGL$ and $\LogicS$, respectively, and sequents with $\seqIII$ correspond to provability in $\LogicD$.

Another well-known result, due to Gödel, McKinsey and Tarski \cite{God33,MT48},
is that provability in propositional intuitionistic logic $\LogicInt$ can be interpreted as provability in the modal logic $\LogicSFour$
via a translation known as Gödel translation.
When such a correspondence holds, the modal logic is called a \emph{modal companion} of the propositional logic
(for instance, $\LogicSFour$ is a modal companion of $\LogicInt$).
It is well known which modal logics are the modal companions of logics stronger than $\LogicInt$, that is, of superintuitionistic logics
(see \cite{CZ97}).
On the other hand, it has also been investigated which propositional logics have modal companions weaker than or incomparable to $\LogicSFour$.
Visser \cite{Vis81} found two propositional logics whose modal companions, via a modified Gödel translation, are $\LogicKFour$ and $\LogicGL$, respectively,
and named them Basic Propositional Logic $\LogicBPL$ and Formal Propositional Logic $\LogicFPL$.
Simple sequent calculi for $\LogicBPL$ and $\LogicFPL$ were proposed by Ishii et al.\ \cite{IKK01}.
Yamasaki and Sano \cite{YS17} modified their sequent calculus for $\LogicBPL$ into a G3-style sequent calculus
and investigated its proof-theoretic properties.
Moreover, they showed the embedding of $\LogicBPL$ into $\LogicKFour$ via the modified Gödel translation syntactically
by means of analytic proofs.
We also note that, in the direction opposite to ours, Chen \cite{Che21} discussed such an embedding for the logic $\Logic{F}$ introduced by Corsi \cite{Cor87}, which is weaker than $\LogicBPL$ and corresponds to the modal logic $\LogicK$.

Petrukhin \cite{Pet23} combined the idea of Kushida's two-level sequent calculus for $\LogicS$ with a sequent calculus for $\LogicFPL$,
proposing a sequent calculus for a propositional logic with two levels of sequents,
and he named the propositional logic obtained from this sequent calculus Solovay Propositional Logic $\LogicSPL$.
Furthermore, by employing the method of Yamasaki and Sano,
he claims that the sequent calculus for $\LogicSPL$ can be syntactically embedded into Kushida's sequent calculus for $\LogicS$ via the modified Gödel translation,
that is, $\LogicS$ is a modal companion of $\LogicSPL$.
However, Petrukhin's proof appears to contain some errors,
and it seems that straightforward corrections are insufficient to fix them.

This paper is devoted to two objectives.
The first is to provide a complete and detailed proof from scratch that $\LogicS$ is the modal companion of $\LogicSPL$,
and in the course of doing so, to identify the errors in Petrukhin's paper.
For the sequent calculus of $\LogicS$, we use not Kushida's original calculus
but a further modification of the one corrected by Kashima et al.\ \cite{KK23}.
The second objective is to define a sequent calculus for a propositional logic,
which we name Dzhaparidze Propositional Logic $\LogicDPL$.
The sequent calculus for $\LogicDPL$ has three levels of sequents,
and is obtained by adding to the sequent calculus for $\LogicSPL$ one rule which arises naturally from the sequent calculus for $\LogicD$.
Then we show, by a slight adaptation of the proof we gave for $\LogicSPL$,
that $\LogicDPL$ can be embedded into $\LogicD$, that is, $\LogicD$ is a modal companion of $\LogicDPL$.

This paper is organized as follows.
In Section \ref{sect:preliminaries}, we introduce some notations and definitions that are used throughout the paper and mention some basic facts.
In Section \ref{sect:GentzenS}, we define a sequent calculus $\GentzenS$ for $\LogicS$
and show some syntactic results from \cite{Kus20,KK23}.
In Section \ref{sect:GentzenSPL}, we define a sequent calculus $\GentzenSPL$ for $\LogicSPL$.
In Section \ref{sect:Embedding_GentzenSPL_GentzenS}, we show that $\GentzenSPL$ can be embedded into $\GentzenS$ via the modified Gödel translation.
Additionally, we mention some errors in Petrukhin's proof in Section \ref{sect:errors_in_Petrukhin}.
In Section \ref{sect:GentzenD}, we define a sequent calculus $\GentzenD$ for $\LogicD$ from \cite{KKIM25}.
In Section \ref{sect:GentzenDPL}, we define a sequent calculus $\GentzenDPL$ for $\LogicDPL$.
In Section \ref{sect:Embedding_GentzenDPL_GentzenD}, we show that $\GentzenDPL$ can be embedded into $\GentzenD$ via the modified Gödel translation
by modifying the proof in Section \ref{sect:Embedding_GentzenSPL_GentzenS}.
Finally, we give some concluding remarks and open problems in Section \ref{sect:conclusion}.

\section{Preliminaries} \label{sect:preliminaries}

In this section, we introduce some notations and definitions that are used throughout the paper and mention some basic facts.
For more details, refer to the standard textbooks such as \cite{Boo94, CZ97}.

Propositional variables are denoted by lowercase letters $p, q, r, \ldots$, and the set of all propositional variables is denoted by $\PropVar$.
Formulas of propositional logic are constructed from propositional variables by the logical connectives $\bot, \land, \lor, \limp$.
Formulas of modal logic are constructed from propositional variables by $\bot, \land, \lor, \limp$ and the modal operator $\Box$.
Other connectives are defined as usual.
Formulas of both propositional logic and modal logic are denoted by uppercase letters $A, B, C, \ldots$.
We write $\FmlPL$ for the set of all formulas of propositional logic, and $\FmlML$ for the set of all formulas of modal logic.
For our purpose, uppercase Greek letters $\Gamma, \Delta, \ldots$ denote \emph{sets} of formulas.
For a set $\Gamma$ of modal formulas, we write $\Box\Gamma$ for the set $\{\Box A : A \in \Gamma\}$.
Moreover, for a finite set $\Gamma$, we write $\lconj \Gamma$ and $\ldisj \Gamma$ for the conjunction and the disjunction of all elements of $\Gamma$,
respectively, where the empty conjunction $\lconj \emptyset$ and the empty disjunction $\ldisj \emptyset$ are understood as $\top$ and $\bot$.

Modal logic $\LogicKFour$ is the smallest normal modal logic containing the axiom $\Box A \limp \Box \Box A$,
and Gödel-Löb provability logic $\LogicGL$ is the smallest normal modal logic containing the axiom $\Box(\Box A \limp A) \limp \Box A$.
Solovay's modal logic $\LogicS$ is a non-normal extension of $\LogicGL$ with the axiom $\Box A \limp A$;
that is, $\LogicS$ is the smallest set of formulas that contains all theorems of $\LogicGL$ and all instances of $\Box A \limp A$, and is closed under modus ponens only (but not under necessitation).
Japaridze's modal logic $\LogicD$ is the non-normal extension of $\LogicGL$ with the axioms $\Box (\Box A \lor \Box B) \limp \Box A \lor \Box B$ and $\lnot \Box \bot$, defined analogously.
The following proposition can be shown by using the semantic tools for $\LogicS$ and $\LogicD$ developed in \cite{Vis84,Bek89}.

\begin{prop}
  $\LogicGL \subset \LogicD \subset \LogicS$.
  That is, $\LogicGL \proves A$ implies $\LogicD \proves A$, and $\LogicD \proves A$ implies $\LogicS \proves A$, but the converses do not hold.
\end{prop}

The Gödel translation is a well-known embedding of intuitionistic propositional logic into modal logic.
However, for our purpose, we need a different translation, introduced by Visser \cite{Vis81}, which we call Visser's Gödel translation.

\begin{defn} \label{defn:visser}
  Visser's Gödel translation $(\cdot)^\Vis : \FmlPL \to \FmlML$ is defined as follows.
  \begin{itemize}
    \item $p^\Vis = p \land \Box p$
    \item $\bot^\Vis = \bot$
    \item $(A \land B)^\Vis = A^\Vis \land B^\Vis$
    \item $(A \lor B)^\Vis = A^\Vis \lor B^\Vis$
    \item $(A \limp B)^\Vis = \Box(A^\Vis \limp B^\Vis)$
  \end{itemize}
\end{defn}

For a set $\Gamma$ of propositional formulas, we write $\Gamma^\Vis$ for the set $\{A^\Vis : A \in \Gamma\}$.

The difference between the original translation and Visser's lies only in the translation of propositional variables:
the original Gödel translation maps $p$ to $\Box p$, while Visser's maps $p$ to $p \land \Box p$.

\begin{prop}[\cite{Vis81}]
  Let $A \in \FmlPL$.
  \begin{itemize}
    \item $\LogicBPL \proves A$ iff $\LogicKFour \proves A^\Vis$.
    \item $\LogicFPL \proves A$ iff $\LogicGL \proves A^\Vis$.
  \end{itemize}
\end{prop}

Finally, we fix the proof-theoretic terminology used throughout this paper.
For the sequent calculi introduced in the following sections,
the notions of a proof-tree and its height, and of (height-preserving) admissible rules,
are defined as usual (see e.g.\ \cite{TS00, NvP11}).
We write $\mathsf{G} \proves S$ if a sequent $S$ is derivable in a sequent calculus $\mathsf{G}$,
and $\mathsf{G} \proves^h S$ if $S$ is derivable in $\mathsf{G}$ by a proof-tree of height at most $h$.
For a formula $A$ and sets $\Gamma, \Delta$ of formulas, in the context of sequents,
we write $A, \Gamma$ for $\{A\} \cup \Gamma$ and $\Gamma, \Delta$ for $\Gamma \cup \Delta$.
Note that $A, \Gamma = \Gamma$ when $A \in \Gamma$, so that contraction is built into this notation.

\section{Sequent calculus $\GentzenS$ for modal logic $\LogicS$} \label{sect:GentzenS}

Although Kushida \cite{Kus20} first introduced a sequent calculus for $\LogicS$, his calculus is not suitable for our purpose.
Here we introduce a sequent calculus $\GentzenS$ obtained by slightly modifying the sequent calculus for $\LogicS$ due to Kashima and Kato \cite{KK23}.
We present below the definition of $\GentzenS$ together with its basic proof-theoretic properties.

\begin{defn}
  Let $\Gamma, \Delta$ be sets of formulas and let $\lev = 1, 2$.
  An expression of the form $\Gamma \seq{\lev} \Delta$ is called a \emph{sequent}.
  The sequent calculus $\GentzenS$ for $\LogicS$, in which sequents of level $\lev = 1, 2$ occur, is defined by the following rules.

  \begin{center}
    \begin{minipage}{0.46\textwidth}\centering
      \begin{prooftree}
        \AxiomC{}
        \RightLabel{$\RAx{\lev}$}
        \UnaryInfC{$p \seq{\lev} p$}
      \end{prooftree}
      where $p \in \PropVar$.
    \end{minipage}\hfill
    \begin{minipage}{0.46\textwidth}\centering
      \begin{prooftree}
        \AxiomC{}
        \RightLabel{$\RBotL{\lev}$}
        \UnaryInfC{$\bot \seq{\lev} {}$}
      \end{prooftree}
    \end{minipage}

    \medskip

    \begin{minipage}{0.46\textwidth}\centering
      \begin{prooftree}
        \AxiomC{$\Gamma \seq{\lev} \Delta$}
        \RightLabel{$\RWL{\lev}{\lev}$}
        \UnaryInfC{$A, \Gamma \seq{\lev} \Delta$}
      \end{prooftree}
    \end{minipage}\hfill
    \begin{minipage}{0.46\textwidth}\centering
      \begin{prooftree}
        \AxiomC{$\Gamma \seq{\lev} \Delta$}
        \RightLabel{$\RWR{\lev}{\lev}$}
        \UnaryInfC{$\Gamma \seq{\lev} \Delta, A$}
      \end{prooftree}
    \end{minipage}

    \medskip

    \begin{prooftree}
      \AxiomC{$\Gamma \seq{1} \Delta$}
      \RightLabel{$\RLift{1}{2}$}
      \UnaryInfC{$\Gamma \seq{2} \Delta$}
    \end{prooftree}

    \medskip

    \begin{minipage}{0.46\textwidth}\centering
      \begin{prooftree}
        \AxiomC{$A, B, \Gamma \seq{\lev} \Delta$}
        \RightLabel{$\RAndL{\lev}{\lev}$}
        \UnaryInfC{$A \land B, \Gamma \seq{\lev} \Delta$}
      \end{prooftree}
    \end{minipage}\hfill
    \begin{minipage}{0.46\textwidth}\centering
      \begin{prooftree}
        \AxiomC{$\Gamma \seq{\lev} \Delta, A$}
        \AxiomC{$\Gamma \seq{\lev} \Delta, B$}
        \RightLabel{$\RAndR{\lev}{\lev}$}
        \BinaryInfC{$\Gamma \seq{\lev} \Delta, A \land B$}
      \end{prooftree}
    \end{minipage}

    \medskip

    \begin{minipage}{0.46\textwidth}\centering
      \begin{prooftree}
        \AxiomC{$A, \Gamma \seq{\lev} \Delta$}
        \AxiomC{$B, \Gamma \seq{\lev} \Delta$}
        \RightLabel{$\ROrL{\lev}{\lev}$}
        \BinaryInfC{$A \lor B, \Gamma \seq{\lev} \Delta$}
      \end{prooftree}
    \end{minipage}\hfill
    \begin{minipage}{0.46\textwidth}\centering
      \begin{prooftree}
        \AxiomC{$\Gamma \seq{\lev} \Delta, A, B$}
        \RightLabel{$\ROrR{\lev}{\lev}$}
        \UnaryInfC{$\Gamma \seq{\lev} \Delta, A \lor B$}
      \end{prooftree}
    \end{minipage}

    \medskip

    \begin{minipage}{0.46\textwidth}\centering
      \begin{prooftree}
        \AxiomC{$\Gamma \seq{\lev} \Delta, A$}
        \AxiomC{$B, \Gamma \seq{\lev} \Delta$}
        \RightLabel{$\RImpL{\lev}{\lev}$}
        \BinaryInfC{$A \limp B, \Gamma \seq{\lev} \Delta$}
      \end{prooftree}
    \end{minipage}\hfill
    \begin{minipage}{0.46\textwidth}\centering
      \begin{prooftree}
        \AxiomC{$A, \Gamma \seq{\lev} \Delta, B$}
        \RightLabel{$\RImpR{\lev}{\lev}$}
        \UnaryInfC{$\Gamma \seq{\lev} \Delta, A \limp B$}
      \end{prooftree}
    \end{minipage}

    \medskip

    \begin{minipage}{0.46\textwidth}\centering
      \begin{prooftree}
        \AxiomC{$\Box\Gamma, \Gamma, \Box A \seq{1} A$}
        \RightLabel{$\RBoxGL{1}{1}$}
        \UnaryInfC{$\Box\Gamma \seq{1} \Box A$}
      \end{prooftree}
    \end{minipage}\hfill
    \begin{minipage}{0.46\textwidth}\centering
      \begin{prooftree}
        \AxiomC{$A, \Gamma \seq{2} \Delta$}
        \RightLabel{$\RBoxL{2}{2}$}
        \UnaryInfC{$\Box A, \Gamma \seq{2} \Delta$}
      \end{prooftree}
    \end{minipage}
  \end{center}
\end{defn}

The provability of $\seq{1}$ is the same as that of the sequent calculus for $\LogicGL$ (cf.\ \cite{SV82}).

\begin{rem} \label{rem:GentzenS_differences}
  Our calculus differs from the system of \cite{KK23} in two respects:
  the axiom $\RAx{\lev}$ is restricted to propositional variables, and the rules for $\land$ and $\lor$ are adopted as primitive rules.
  Nevertheless, by suitably adapting the semantic arguments of \cite{KK23}, one can verify that the facts stated below,
  such as cut-admissibility and the coincidence of provability with that of $\LogicS$, remain valid for our calculus.
\end{rem}

In particular, the identity sequents for arbitrary formulas, which are the axioms of the system of \cite{KK23}, are derivable.

\begin{lem} \label{lem:GentzenS_original_axioms}
  For every $A \in \FmlML$, $\GentzenS \proves A \seq{\lev} A$.
\end{lem}

\begin{proof}
  We show by induction on $A$, uniformly in $\lev = 1, 2$.
  If $A \equiv p$, the claim follows from $\RAx{\lev}$.
  If $A \equiv \bot$, it follows from $\RBotL{\lev}$ and $\RWR{\lev}{\lev}$.
  The cases $A \equiv B \land C$ and $A \equiv B \lor C$ follow from the induction hypothesis by the rules for $\land$ and $\lor$ together with weakening, as usual.
  The case $A \equiv B \limp C$ is proved by the following proof-tree:
  \begin{prooftree}
    \AxiomC{\scriptsize(I.H.)}
    \noLine
    \UnaryInfC{$B \seq{\lev} B$}
    \RightLabel{$\RWR{\lev}{\lev}$}
    \UnaryInfC{$B \seq{\lev} C, B$}
    \AxiomC{\scriptsize(I.H.)}
    \noLine
    \UnaryInfC{$C \seq{\lev} C$}
    \RightLabel{$\RWL{\lev}{\lev}$}
    \UnaryInfC{$C, B \seq{\lev} C$}
    \RightLabel{$\RImpL{\lev}{\lev}$}
    \BinaryInfC{$B \limp C, B \seq{\lev} C$}
    \RightLabel{$\RImpR{\lev}{\lev}$}
    \UnaryInfC{$B \limp C \seq{\lev} B \limp C$}
  \end{prooftree}
  The case $A \equiv \Box B$ is proved by the following proof-tree:
  \begin{prooftree}
    \AxiomC{\scriptsize(I.H.)}
    \noLine
    \UnaryInfC{$B \seq{1} B$}
    \RightLabel{$\RWL{1}{1}$}
    \UnaryInfC{$B, \Box B \seq{1} B$}
    \RightLabel{$\RBoxGL{1}{1}$}
    \UnaryInfC{$\Box B \seq{1} \Box B$}
  \end{prooftree}
  When $\lev = 2$, we apply $\RLift{1}{2}$ to obtain $\Box B \seq{2} \Box B$.
\end{proof}

\begin{prop}[Provability of $\GentzenS$ \cite{KK23}] \label{prop:GentzenS_provability}
  The following equivalences hold.
  \begin{enumerate}
    \item $\GentzenS \proves \Gamma \seq{1} \Delta$ iff $\LogicGL \proves \lconj \Gamma \limp \ldisj \Delta$.
    \item $\GentzenS \proves \Gamma \seq{2} \Delta$ iff $\LogicS \proves \lconj \Gamma \limp \ldisj \Delta$.
  \end{enumerate}
\end{prop}

\begin{prop}[Cut-admissibility of $\GentzenS$ {\cite[Theorem 3.1]{KK23}}] \label{prop:GentzenS_cut_elimination}
  The cut rule is admissible in $\GentzenS$.
  That is, when the cut rule $\RCut{\lev}{\lev}$ is defined as below, if $\GentzenS + \RCut{\lev}{\lev} \proves \Gamma \seq{\lev} \Delta$ then $\GentzenS \proves \Gamma \seq{\lev} \Delta$.
  \begin{prooftree}
    \AxiomC{$\Gamma_1 \seq{\lev} \Delta_1, A$}
    \AxiomC{$A, \Gamma_2 \seq{\lev} \Delta_2$}
    \RightLabel{$\RCut{\lev}{\lev}$}
    \BinaryInfC{$\Gamma_1, \Gamma_2 \seq{\lev} \Delta_1, \Delta_2$}
  \end{prooftree}
\end{prop}

\begin{lem}[Inversion rules $\Inv{\RImpL{\lev}{\lev}}$] \label{lem:GentzenS_inversion_ImpL}
  In $\GentzenS$ the inversion rules $\Inv{\RImpL{\lev}{\lev}}$ are height-preservingly admissible.
  \begin{center}
    \begin{minipage}{0.46\textwidth}\centering
      \begin{prooftree}
        \AxiomC{$\GentzenS \proves^h A \limp B, \Gamma \seq{\lev} \Delta$}
        \RightLabel{$\Inv{\RImpL{\lev}{\lev}}$}
        \UnaryInfC{$\GentzenS \proves^h \Gamma \seq{\lev} \Delta, A$}
      \end{prooftree}
    \end{minipage}\hfill
    \begin{minipage}{0.46\textwidth}\centering
      \begin{prooftree}
        \AxiomC{$\GentzenS \proves^h A \limp B, \Gamma \seq{\lev} \Delta$}
        \RightLabel{$\Inv{\RImpL{\lev}{\lev}}$}
        \UnaryInfC{$\GentzenS \proves^h B, \Gamma \seq{\lev} \Delta$}
      \end{prooftree}
    \end{minipage}
  \end{center}
\end{lem}

\begin{proof}
  We prove the two claims simultaneously by induction on the height of the proof-tree.
  In order to derive $A \limp B, \Gamma \seq{\lev} \Delta$, the rules in which $A \limp B$ can occur as the principal formula are the two rules $\RWL{\lev}{\lev}$ and $\RImpL{\lev}{\lev}$.
  For the other rules, and when $A \limp B$ is not principal, the claim follows immediately from the induction hypothesis.
  The case of $\RImpL{\lev}{\lev}$ is also clear.
  In the case of $\RWL{\lev}{\lev}$ we have $\GentzenS \proves^{h-1} \Gamma \seq{\lev} \Delta$, so by means of $\RWR{\lev}{\lev}$ or $\RWL{\lev}{\lev}$ we obtain $\GentzenS \proves^h \Gamma \seq{\lev} \Delta, A$ and $\GentzenS \proves^h B, \Gamma \seq{\lev} \Delta$.
\end{proof}

\begin{lem}[Inversion rule $\Inv{\RImpR{\lev}{\lev}}$] \label{lem:GentzenS_inversion_ImpR}
  In $\GentzenS$ the inversion rule $\Inv{\RImpR{\lev}{\lev}}$ is height-preservingly admissible.
  \begin{prooftree}
    \AxiomC{$\GentzenS \proves^h \Gamma \seq{\lev} \Delta, A \limp B$}
    \RightLabel{$\Inv{\RImpR{\lev}{\lev}}$}
    \UnaryInfC{$\GentzenS \proves^h A, \Gamma \seq{\lev} \Delta, B$}
  \end{prooftree}
\end{lem}

\begin{proof}
  We show by induction on the height of the proof-tree.
  In order to derive $\Gamma \seq{\lev} \Delta, A \limp B$, the rules in which $A \limp B$ can occur as the principal formula are the two rules $\RWR{\lev}{\lev}$ and $\RImpR{\lev}{\lev}$.
  For the other rules, and when $A \limp B$ is not principal, the claim follows immediately from the induction hypothesis.
  These two rules are handled in the same way as in the case of $\Inv{\RImpL{\lev}{\lev}}$.
\end{proof}

\section{Sequent calculus $\GentzenSPL$ for propositional logic $\LogicSPL$} \label{sect:GentzenSPL}

We define $\GentzenSPL$, a modification of the sequent calculus for $\LogicSPL$ due to Petrukhin \cite{Pet23}.

\begin{defn}
  Take natural numbers $k \geq 0$ and $n < 2^k$, and let $n = \sum_{i=0}^{k-1} 2^{i} c_i$ be the $k$-bit binary representation of $n$.
  We then set $\langle n \rangle^k_i = c_i$.
  To be explicit, $\langle 0 \rangle^0_0 = 0$.
  Moreover, regarding $n$ as a $k$-bit string, we write $\bar{n}$ for its one's complement.
  That is, for $k \geq 1$ we have $\langle \bar{n} \rangle^k_i = 1 - \langle n \rangle^k_i$, and $\langle \bar{0} \rangle^0_0 = 0$.
\end{defn}

\begin{defn}
  Let $\Gamma, \Delta$ be sets of formulas and let $\lev = 1, 2$.
  An expression of the form $\Gamma \seq{\lev} \Delta$ is called a \emph{sequent}.
  The sequent calculus $\GentzenSPL$ for $\LogicSPL$, in which sequents of level $\lev = 1, 2$ occur, is defined by the following rules.

  \begin{center}
    \begin{minipage}{0.46\textwidth}\centering
      \begin{prooftree}
        \AxiomC{}
        \RightLabel{$\RAx{\lev}$}
        \UnaryInfC{$p \seq{\lev} p$}
      \end{prooftree}
      where $p \in \PropVar$.
    \end{minipage}\hfill
    \begin{minipage}{0.46\textwidth}\centering
      \begin{prooftree}
        \AxiomC{}
        \RightLabel{$\RBotL{\lev}$}
        \UnaryInfC{$\bot \seq{\lev} {}$}
      \end{prooftree}
    \end{minipage}

    \medskip

    \begin{minipage}{0.46\textwidth}\centering
      \begin{prooftree}
        \AxiomC{$\Gamma \seq{\lev} \Delta$}
        \RightLabel{$\RWL{\lev}{\lev}$}
        \UnaryInfC{$A, \Gamma \seq{\lev} \Delta$}
      \end{prooftree}
    \end{minipage}\hfill
    \begin{minipage}{0.46\textwidth}\centering
      \begin{prooftree}
        \AxiomC{$\Gamma \seq{\lev} \Delta$}
        \RightLabel{$\RWR{\lev}{\lev}$}
        \UnaryInfC{$\Gamma \seq{\lev} \Delta, A$}
      \end{prooftree}
    \end{minipage}

    \medskip

    \begin{prooftree}
      \AxiomC{$\Gamma \seq{1} \Delta$}
      \RightLabel{$\RLift{1}{2}$}
      \UnaryInfC{$\Gamma \seq{2} \Delta$}
    \end{prooftree}

    \medskip

    \begin{minipage}{0.46\textwidth}\centering
      \begin{prooftree}
        \AxiomC{$A, B, \Gamma \seq{\lev} \Delta$}
        \RightLabel{$\RAndL{\lev}{\lev}$}
        \UnaryInfC{$A \land B, \Gamma \seq{\lev} \Delta$}
      \end{prooftree}
    \end{minipage}\hfill
    \begin{minipage}{0.46\textwidth}\centering
      \begin{prooftree}
        \AxiomC{$\Gamma \seq{\lev} \Delta, A$}
        \AxiomC{$\Gamma \seq{\lev} \Delta, B$}
        \RightLabel{$\RAndR{\lev}{\lev}$}
        \BinaryInfC{$\Gamma \seq{\lev} \Delta, A \land B$}
      \end{prooftree}
    \end{minipage}

    \medskip

    \begin{minipage}{0.46\textwidth}\centering
      \begin{prooftree}
        \AxiomC{$A, \Gamma \seq{\lev} \Delta$}
        \AxiomC{$B, \Gamma \seq{\lev} \Delta$}
        \RightLabel{$\ROrL{\lev}{\lev}$}
        \BinaryInfC{$A \lor B, \Gamma \seq{\lev} \Delta$}
      \end{prooftree}
    \end{minipage}\hfill
    \begin{minipage}{0.46\textwidth}\centering
      \begin{prooftree}
        \AxiomC{$\Gamma \seq{\lev} \Delta, A, B$}
        \RightLabel{$\ROrR{\lev}{\lev}$}
        \UnaryInfC{$\Gamma \seq{\lev} \Delta, A \lor B$}
      \end{prooftree}
    \end{minipage}

    \medskip

    \begin{prooftree}
      \AxiomC{$\Gamma \seq{2} \Delta, A$}
      \AxiomC{$B, \Gamma \seq{2} \Delta$}
      \RightLabel{$\RImpL{2}{2}$}
      \BinaryInfC{$A \limp B, \Gamma \seq{2} \Delta$}
    \end{prooftree}

    \medskip

    \begin{prooftree}
      \AxiomC{$\Delta_i, \Sigma, A \limp B, A \seq{1} B, \Gamma_i : 0 \leq i < 2^k$}
      \RightLabel{$\RImpR{1}{1}$}
      \UnaryInfC{$\Sigma, \{C_j \limp D_j : 0 \le j < k\} \seq{1} A \limp B$}
    \end{prooftree}

    where $k \geq 0$ and the formulas $C_0, \dots, C_{k-1}$ and $D_0, \dots, D_{k-1}$ are arbitrary.
    Moreover, for $0 \leq i < 2^k$ we set $\Delta_i := \{D_j : \langle i \rangle^k_j = 1\}$ and $\Gamma_i := \{C_j : \langle \bar{i} \rangle^k_j = 1\}$.
    The expression displayed as the premise of $\RImpR{1}{1}$ denotes the family of the $2^k$ sequents obtained by letting $i$ range over $0 \leq i < 2^k$; that is, the rule $\RImpR{1}{1}$ has $2^k$ premises.
  \end{center}
\end{defn}

Note that $\RImpR{1}{1}$ corresponds to the rule in the sequent calculus for $\LogicFPL$ given in \cite{IKK01}.

\begin{ex} \label{ex:GentzenSPL_ImpR}
  The rule $\RImpR{1}{1}$ for $k = 0, 1, 2$ takes the following forms.

  \medskip\noindent
  $k = 0$:
  \begin{prooftree}
    \AxiomC{$\Sigma, A \limp B, A \seq{1} B$}
    \RightLabel{$\RImpR{1}{1}$}
    \UnaryInfC{$\Sigma \seq{1} A \limp B$}
  \end{prooftree}

  \medskip\noindent
  $k = 1$:
  \begin{prooftree}
    \AxiomC{$D_0, \Sigma, A \limp B, A \seq{1} B$}
    \AxiomC{$\Sigma, A \limp B, A \seq{1} B, C_0$}
    \RightLabel{$\RImpR{1}{1}$}
    \BinaryInfC{$\Sigma, C_0 \limp D_0 \seq{1} A \limp B$}
  \end{prooftree}

  \medskip\noindent
  $k = 2$:
  \pagecenter{\scalebox{0.85}{%
      \AxiomC{$D_0, D_1, \Sigma, A \limp B, A \seq{1} B$}
      \AxiomC{$D_0, \Sigma, A \limp B, A \seq{1} B, C_1$}
      \AxiomC{$D_1, \Sigma, A \limp B, A \seq{1} B, C_0$}
      \AxiomC{$\Sigma, A \limp B, A \seq{1} B, C_0, C_1$}
      \RightLabel{$\RImpR{1}{1}$}
      \QuaternaryInfC{$\Sigma, C_0 \limp D_0, C_1 \limp D_1 \seq{1} A \limp B$}
      \DisplayProof
    }}
\end{ex}

As in Section \ref{sect:GentzenS}, although the rule $\RAx{\lev}$ is stated only for propositional variables, the identity sequent is derivable for every formula.

\begin{lem} \label{lem:GentzenSPL_axioms}
  For every $A \in \FmlPL$, $\GentzenSPL \proves A \seq{\lev} A$.
\end{lem}

\begin{proof}
  We show by induction on $A$, uniformly in $\lev = 1, 2$.
  If $A \equiv p$, the claim follows from $\RAx{\lev}$.
  If $A \equiv \bot$, it follows from $\RBotL{\lev}$ and $\RWR{\lev}{\lev}$.
  The cases $A \equiv B \land C$ and $A \equiv B \lor C$ follow from the induction hypothesis by the rules for $\land$ and $\lor$ together with weakening, as usual.
  For the case $A \equiv B \limp C$, we first derive $B \limp C \seq{1} B \limp C$ by the following instance of $\RImpR{1}{1}$ for $k = 1$ with $\Sigma = \emptyset$ and $C_0 \limp D_0 := B \limp C$.
  \begin{prooftree}
    \AxiomC{\scriptsize(I.H.)}
    \noLine
    \UnaryInfC{$C \seq{1} C$}
    \RightLabel{$\RWL{1}{1}$}
    \UnaryInfC{$C, B \limp C, B \seq{1} C$}
    \AxiomC{\scriptsize(I.H.)}
    \noLine
    \UnaryInfC{$B \seq{1} B$}
    \RightLabel{$\RWL{1}{1}$, $\RWR{1}{1}$}
    \UnaryInfC{$B \limp C, B \seq{1} C, B$}
    \RightLabel{$\RImpR{1}{1}$}
    \BinaryInfC{$B \limp C \seq{1} B \limp C$}
  \end{prooftree}
  When $\lev = 2$, we further apply $\RLift{1}{2}$.
\end{proof}

We define \emph{Solovay Propositional Logic} $\LogicSPL$ as the set of formulas $\{A \in \FmlPL : \GentzenSPL \proves \seqII A\}$.

\section{Embedding of $\GentzenSPL$ into $\GentzenS$} \label{sect:Embedding_GentzenSPL_GentzenS}

We give a detailed proof of the embedding of $\GentzenSPL$ into $\GentzenS$ based on the proofs in \cite{YS17, Pet23}.

\begin{lem} \label{lem:GentzenS_admits_BoxK4}
  The following rule $\RBoxKFour{1}{\lev}$ is admissible in $\GentzenS$.
  \begin{prooftree}
    \AxiomC{$\Box\Gamma, \Gamma \seq{1} A$}
    \RightLabel{$\RBoxKFour{1}{\lev}$}
    \UnaryInfC{$\Box\Gamma \seq{\lev} \Box A$}
  \end{prooftree}
\end{lem}

\begin{proof}
  For $\lev = 1$, the claim is proved by the following proof-tree:
  \begin{prooftree}
    \AxiomC{$\Gamma, \Box\Gamma \seq{1} A$}
    \RightLabel{$\RWL{1}{1}$}
    \UnaryInfC{$\Gamma, \Box\Gamma, \Box A \seq{1} A$}
    \RightLabel{$\RBoxGL{1}{1}$}
    \UnaryInfC{$\Box\Gamma \seq{1} \Box A$}
  \end{prooftree}
  When $\lev = 2$, we further apply $\RLift{1}{2}$ to obtain $\Box\Gamma \seq{2} \Box A$.
\end{proof}

\begin{lem}[cf.\ {\cite[Lemma 3]{YS17}}] \label{lem:embedding_SPL_to_S_aux1}
  For every $A \in \FmlPL$, $\GentzenS \proves A^\Vis \seq{\lev} \Box A^\Vis$.
\end{lem}

\begin{proof}
  We show by induction on $A$. The cases $A \equiv \bot$ and $A \equiv B \lor C$ are omitted.

  \medskip\noindent
  \textbf{Case} $A \equiv p$:
  \begin{prooftree}
    \AxiomC{\scriptsize(Lemma \ref{lem:GentzenS_original_axioms}, $\RWL{1}{1}$)}
    \noLine
    \UnaryInfC{$\Box p, p \seq{1} p$}
    \AxiomC{\scriptsize(Lemma \ref{lem:GentzenS_original_axioms}, $\RWL{1}{1}$)}
    \noLine
    \UnaryInfC{$\Box p, p \seq{1} \Box p$}
    \RightLabel{$\RAndR{1}{1}$}
    \BinaryInfC{$\Box p, p \seq{1} p \land \Box p$}
    \RightLabel{$\RBoxKFour{1}{\lev}$}
    \UnaryInfC{$\Box p \seq{\lev} \Box(p \land \Box p)$}
    \RightLabel{$\RWL{\lev}{\lev}$}
    \UnaryInfC{$p, \Box p \seq{\lev} \Box(p \land \Box p)$}
    \RightLabel{$\RAndL{\lev}{\lev}$}
    \UnaryInfC{$p \land \Box p \seq{\lev} \Box(p \land \Box p)$}
  \end{prooftree}

  \medskip\noindent
  \textbf{Case} $A \equiv B \land C$:
  \pagecenter{%
    \AxiomC{\scriptsize(I.H.)}
    \noLine
    \UnaryInfC{$B^\Vis \seq{\lev} \Box B^\Vis$}
    \RightLabel{$\RWL{\lev}{\lev}$}
    \UnaryInfC{$B^\Vis, C^\Vis \seq{\lev} \Box B^\Vis$}
    \AxiomC{\scriptsize(I.H.)}
    \noLine
    \UnaryInfC{$C^\Vis \seq{\lev} \Box C^\Vis$}
    \RightLabel{$\RWL{\lev}{\lev}$}
    \UnaryInfC{$B^\Vis, C^\Vis \seq{\lev} \Box C^\Vis$}
    \RightLabel{$\RAndR{\lev}{\lev}$}
    \BinaryInfC{$B^\Vis, C^\Vis \seq{\lev} \Box B^\Vis \land \Box C^\Vis$}
    \AxiomC{\scriptsize(Lemma \ref{lem:GentzenS_original_axioms}, $\RWL{1}{1}$)}
    \noLine
    \UnaryInfC{$B^\Vis, C^\Vis, \Box B^\Vis, \Box C^\Vis \seq{1} B^\Vis$}
    \AxiomC{\scriptsize(Lemma \ref{lem:GentzenS_original_axioms}, $\RWL{1}{1}$)}
    \noLine
    \UnaryInfC{$B^\Vis, C^\Vis, \Box B^\Vis, \Box C^\Vis \seq{1} C^\Vis$}
    \RightLabel{$\RAndR{1}{1}$}
    \BinaryInfC{$B^\Vis, C^\Vis, \Box B^\Vis, \Box C^\Vis \seq{1} B^\Vis \land C^\Vis$}
    \RightLabel{$\RBoxKFour{1}{\lev}$}
    \UnaryInfC{$\Box B^\Vis, \Box C^\Vis \seq{\lev} \Box(B^\Vis \land C^\Vis)$}
    \RightLabel{$\RAndL{\lev}{\lev}$}
    \UnaryInfC{$\Box B^\Vis \land \Box C^\Vis \seq{\lev} \Box(B^\Vis \land C^\Vis)$}
    \RightLabel{$\RCut{\lev}{\lev}$}
    \BinaryInfC{$B^\Vis, C^\Vis \seq{\lev} \Box(B^\Vis \land C^\Vis)$}
    \RightLabel{$\RAndL{\lev}{\lev}$}
    \UnaryInfC{$B^\Vis \land C^\Vis \seq{\lev} \Box(B^\Vis \land C^\Vis)$}
    \DisplayProof
  }
  Here the use of the cut rule $\RCut{\lev}{\lev}$ is legitimate by Proposition \ref{prop:GentzenS_cut_elimination}; we use the cut rule freely in proof-trees in what follows.

  \medskip\noindent
  \textbf{Case} $A \equiv B \limp C$:
  \begin{prooftree}
    \AxiomC{\scriptsize(Lemma \ref{lem:GentzenS_original_axioms}, $\RWL{1}{1}$)}
    \noLine
    \UnaryInfC{$\Box(B^\Vis \limp C^\Vis), B^\Vis \limp C^\Vis \seq{1} \Box(B^\Vis \limp C^\Vis)$}
    \RightLabel{$\RBoxKFour{1}{\lev}$}
    \UnaryInfC{$\Box(B^\Vis \limp C^\Vis) \seq{\lev} \Box\Box(B^\Vis \limp C^\Vis)$}
  \end{prooftree}
\end{proof}

\begin{lem}[cf.\ {\cite[Lemma 5]{YS17}}] \label{lem:embedding_SPL_to_S_aux2}
  The following rule is admissible in $\GentzenS$.
  \begin{prooftree}
    \AxiomC{$\Sigma^\Vis, \Gamma, \Box A \seq{1} A$}
    \UnaryInfC{$\Sigma^\Vis, \Box\Gamma \seq{\lev} \Box A$}
  \end{prooftree}
\end{lem}

\begin{proof}
  First, we obtain $\Box\Sigma^\Vis, \Box\Gamma \seq{1} \Box A$ as follows.
  \begin{prooftree}
    \AxiomC{$\Sigma^\Vis, \Gamma, \Box A \seq{1} A$}
    \RightLabel{$\RWL{1}{1}$}
    \UnaryInfC{$\Box\Sigma^\Vis, \Sigma^\Vis, \Box\Gamma, \Gamma, \Box A \seq{1} A$}
    \RightLabel{$\RBoxGL{1}{1}$}
    \UnaryInfC{$\Box\Sigma^\Vis, \Box\Gamma \seq{1} \Box A$}
  \end{prooftree}
  Now, by Lemma \ref{lem:embedding_SPL_to_S_aux1}, for every $B \in \Sigma$ we have $\GentzenS \proves B^\Vis \seq{1} \Box B^\Vis$.
  Hence, by applying $\RCut{1}{1}$ an appropriate number of times to successively replace the elements of $\Box\Sigma^\Vis$, we obtain $\Sigma^\Vis, \Box\Gamma \seq{1} \Box A$.
  When $\lev = 2$, we further apply $\RLift{1}{2}$.
\end{proof}

\begin{lem}[{\cite[Lemma 4]{Pet23}}] \label{lem:embedding_SPL_to_S}
  If $\GentzenSPL \proves \Gamma \seq{\lev} \Delta$, then $\GentzenS \proves \Gamma^\Vis \seq{\lev} \Delta^\Vis$.
\end{lem}

\begin{proof}
  We show by induction on the height of the proof-tree.
  We treat in particular the cases of $\RImpL{2}{2}$ and $\RImpR{1}{1}$.

  \medskip\noindent
  \textbf{Case} $\RImpL{2}{2}$:
  Suppose the rule we now wish to consider has the following form.
  \begin{prooftree}
    \AxiomC{$\GentzenSPL \proves^{h-1} \Gamma \seq{2} \Delta, A$}
    \AxiomC{$\GentzenSPL \proves^{h-1} B, \Gamma \seq{2} \Delta$}
    \RightLabel{$\RImpL{2}{2}$}
    \BinaryInfC{$\GentzenSPL \proves^{h} A \limp B, \Gamma \seq{2} \Delta$}
  \end{prooftree}
  Then, by the induction hypothesis, we have $\GentzenS \proves \Gamma^\Vis \seq{2} \Delta^\Vis, A^\Vis$ and $\GentzenS \proves B^\Vis, \Gamma^\Vis \seq{2} \Delta^\Vis$.
  Given this, we proceed as follows.
  \begin{prooftree}
    \AxiomC{$\Gamma^\Vis \seq{2} \Delta^\Vis, A^\Vis$}
    \AxiomC{$B^\Vis, \Gamma^\Vis \seq{2} \Delta^\Vis$}
    \RightLabel{$\RImpL{2}{2}$}
    \BinaryInfC{$A^\Vis \limp B^\Vis, \Gamma^\Vis \seq{2} \Delta^\Vis$}
    \RightLabel{$\RBoxL{2}{2}$}
    \UnaryInfC{$\Box(A^\Vis \limp B^\Vis), \Gamma^\Vis \seq{2} \Delta^\Vis$}
  \end{prooftree}
  This is nothing but $\GentzenS \proves (A \limp B)^\Vis, \Gamma^\Vis \seq{2} \Delta^\Vis$.

  \medskip\noindent
  \textbf{Case} $\RImpR{1}{1}$:
  Fix $k$, and suppose the rule we now wish to consider has the following form.
  \begin{prooftree}
    \AxiomC{$\GentzenSPL \proves^{h-1} \Delta_i, \Sigma, A \limp B, A \seq{1} B, \Gamma_i : 0 \leq i < 2^k$}
    \RightLabel{$\RImpR{1}{1}$}
    \UnaryInfC{$\GentzenSPL \proves^{h} \Sigma, \{C_j \limp D_j : 0 \le j < k\} \seq{1} A \limp B$}
  \end{prooftree}
  By the induction hypothesis, for every $i$ with $0 \leq i < 2^k$ we have $\GentzenS \proves \Delta_i^\Vis, \Sigma^\Vis, \Box(A^\Vis \limp B^\Vis), A^\Vis \seq{1} B^\Vis, \Gamma_i^\Vis$.
  By combining these sequents with $2^k - 1$ appropriate applications of $\RImpL{1}{1}$, we obtain
  $\GentzenS \proves \Sigma^\Vis, \{C_j^\Vis \limp D_j^\Vis : 0 \le j < k\}, \Box(A^\Vis \limp B^\Vis), A^\Vis \seq{1} B^\Vis$.
  We illustrate the case $k = 2$ below; the construction is easy to generalize to arbitrary $k$.
  Let $\mathcal{D}_1$ and $\mathcal{D}_2$ be the following proof-trees.
  \pagecenter{%
    \AxiomC{$\Sigma^\Vis, \Box(A^\Vis \limp B^\Vis), A^\Vis \seq{1} B^\Vis, C_0^\Vis, C_1^\Vis$}
    \AxiomC{$D_0^\Vis, \Sigma^\Vis, \Box(A^\Vis \limp B^\Vis), A^\Vis \seq{1} B^\Vis, C_1^\Vis$}
    \RightLabel{$\RImpL{1}{1}$}
    \BinaryInfC{$\Sigma^\Vis, C_0^\Vis \limp D_0^\Vis, \Box(A^\Vis \limp B^\Vis), A^\Vis \seq{1} B^\Vis, C_1^\Vis$}
    \DisplayProof
  }
  \pagecenter{%
    \AxiomC{$D_1^\Vis, \Sigma^\Vis, \Box(A^\Vis \limp B^\Vis), A^\Vis \seq{1} B^\Vis, C_0^\Vis$}
    \AxiomC{$D_0^\Vis, D_1^\Vis, \Sigma^\Vis, \Box(A^\Vis \limp B^\Vis), A^\Vis \seq{1} B^\Vis$}
    \RightLabel{$\RImpL{1}{1}$}
    \BinaryInfC{$D_1^\Vis, \Sigma^\Vis, C_0^\Vis \limp D_0^\Vis, \Box(A^\Vis \limp B^\Vis), A^\Vis \seq{1} B^\Vis$}
    \DisplayProof
  }
  Then, applying $\RImpL{1}{1}$ once more to the end sequents of $\mathcal{D}_1$ and $\mathcal{D}_2$, we obtain the desired sequent.
  \pagecenter{%
    \AxiomC{$\mathcal{D}_1$}
    \noLine
    \UnaryInfC{$\Sigma^\Vis, C_0^\Vis \limp D_0^\Vis, \Box(A^\Vis \limp B^\Vis), A^\Vis \seq{1} B^\Vis, C_1^\Vis$}
    \AxiomC{$\mathcal{D}_2$}
    \noLine
    \UnaryInfC{$D_1^\Vis, \Sigma^\Vis, C_0^\Vis \limp D_0^\Vis, \Box(A^\Vis \limp B^\Vis), A^\Vis \seq{1} B^\Vis$}
    \RightLabel{$\RImpL{1}{1}$}
    \BinaryInfC{$\Sigma^\Vis, C_0^\Vis \limp D_0^\Vis, C_1^\Vis \limp D_1^\Vis, \Box(A^\Vis \limp B^\Vis), A^\Vis \seq{1} B^\Vis$}
    \DisplayProof
  }

  The rest can be carried out as follows.
  \pagecenter{%
    \AxiomC{$\Sigma^\Vis, \{C_j^\Vis \limp D_j^\Vis : 0 \le j < k\}, \Box(A^\Vis \limp B^\Vis), A^\Vis \seq{1} B^\Vis$}
    \RightLabel{$\RImpR{1}{1}$}
    \UnaryInfC{$\Sigma^\Vis, \{C_j^\Vis \limp D_j^\Vis : 0 \le j < k\}, \Box(A^\Vis \limp B^\Vis) \seq{1} A^\Vis \limp B^\Vis$}
    \RightLabel{Lemma \ref{lem:embedding_SPL_to_S_aux2}}
    \UnaryInfC{$\Sigma^\Vis, \{\Box(C_j^\Vis \limp D_j^\Vis) : 0 \le j < k\} \seq{1} \Box(A^\Vis \limp B^\Vis)$}
    \DisplayProof
  }
  This is nothing but $\GentzenS \proves \Sigma^\Vis, \{C_j \limp D_j : 0 \le j < k\}^\Vis \seq{1} (A \limp B)^\Vis$.
\end{proof}

Following the strategy of {\cite[Lemma 7]{YS17}}, we prove the converse direction in the following strengthened form.

\begin{lem} \label{lem:embedding_S_to_SPL_general}
  Let $\Gamma, \Delta$ be sets of propositional formulas and let $\Phi_1, \Phi_2, \Psi$ be sets of propositional variables.
  If $\GentzenS \proves \Gamma^\Vis, \Phi_1, \Box\Phi_2 \seq{\lev} \Psi, \Delta^\Vis$,
  then $\GentzenSPL \proves \Gamma, \Phi_1, \Phi_2 \seq{\lev} \Psi, \Delta$.
\end{lem}

\begin{proof}
  First, observe that every formula occurring in the assumed sequent is of exactly one of the following forms:
  a translation $A^\Vis$, which is either $\bot$ or has $\land$, $\lor$ or $\Box$ as its outermost connective;
  a propositional variable; or a boxed propositional variable.
  In particular, no formula in the sequent is an implication,
  and hence the last rule of a proof-tree of the assumed sequent is neither $\RImpL{\lev}{\lev}$ nor $\RImpR{\lev}{\lev}$.
  With this in mind, we prove the claim by induction on the height of the proof-tree,
  distinguishing cases according to the last rule and the position of its principal formula.

  Note that the statement is generalized over $\Gamma, \Delta, \Phi_1, \Phi_2$ and $\Psi$:
  in each case below, the premise of the last rule is re-decomposed into the form $\Gamma^\Vis, \Phi_1, \Box\Phi_2 \seq{\lev} \Psi, \Delta^\Vis$, possibly with components different from those of the conclusion, and the induction hypothesis is applied to this decomposition.
  We illustrate this point concretely in the case $\RAndL{\lev}{\lev}$ below.

  \medskip\noindent
  \textbf{Case} $\RAx{\lev}$:
  The sequent is $p \seq{\lev} p$, that is, $\Gamma = \Delta = \Phi_2 = \emptyset$ and $\Phi_1 = \Psi = \{p\}$.
  The rule $\RAx{\lev}$ of $\GentzenSPL$ yields the desired $\GentzenSPL \proves p \seq{\lev} p$.

  \medskip\noindent
  \textbf{Case} $\RBotL{\lev}$:
  The sequent is $\bot \seq{\lev} {}$, that is, $\Gamma = \{\bot\}$ and all the other components are empty.
  The rule $\RBotL{\lev}$ of $\GentzenSPL$ yields the desired sequent.

  \medskip\noindent
  \textbf{Case} $\RWL{\lev}{\lev}$, $\RWR{\lev}{\lev}$, $\RLift{1}{2}$:
  These follow straightforwardly from the induction hypothesis and the corresponding rule of $\GentzenSPL$.
  We only note that no weakening is needed in $\GentzenSPL$ when $\RWL{\lev}{\lev}$ introduces a variable $p \in \Phi_1$ with $p \in \Phi_2$, or a boxed variable $\Box p$ with $p \in \Phi_1$:
  in these cases the antecedent of the sequent given by the induction hypothesis already coincides with $\Gamma, \Phi_1, \Phi_2$ as a set.

  \medskip\noindent
  \textbf{Case} $\RAndL{\lev}{\lev}$:
  The principal formula is either $p^\Vis = p \land \Box p$ for some $p \in \Gamma$, or $(A \land B)^\Vis = A^\Vis \land B^\Vis$ for some $A \land B \in \Gamma$.
  In the former case, the application we are concerned with is the following.
  \pagecenter{%
    \AxiomC{$\GentzenS \proves^{h-1} p, \Box p, (\Gamma \setminus \{p\})^\Vis, \Phi_1, \Box\Phi_2 \seq{\lev} \Psi, \Delta^\Vis$}
    \RightLabel{$\RAndL{\lev}{\lev}$}
    \UnaryInfC{$\GentzenS \proves^{h} p \land \Box p, (\Gamma \setminus \{p\})^\Vis, \Phi_1, \Box\Phi_2 \seq{\lev} \Psi, \Delta^\Vis$}
    \DisplayProof
  }
  Here we re-decompose the premise:
  it is of the form $\Gamma'^\Vis, \Phi_1', \Box\Phi_2' \seq{\lev} \Psi, \Delta^\Vis$ with $\Gamma' := \Gamma \setminus \{p\}$, $\Phi_1' := \Phi_1 \cup \{p\}$ and $\Phi_2' := \Phi_2 \cup \{p\}$,
  which differ from the components $\Gamma, \Phi_1, \Phi_2$ of the conclusion.
  Applying the induction hypothesis to this decomposition yields $\GentzenSPL \proves \Gamma \setminus \{p\}, \Phi_1, \Phi_2, p \seq{\lev} \Psi, \Delta$,
  whose antecedent coincides with $\Gamma, \Phi_1, \Phi_2$ as a set, since $p \in \Gamma$.

  In the latter case, the application we are concerned with is the following.
  \pagecenter{%
    \AxiomC{$\GentzenS \proves^{h-1} A^\Vis, B^\Vis, (\Gamma \setminus \{A \land B\})^\Vis, \Phi_1, \Box\Phi_2 \seq{\lev} \Psi, \Delta^\Vis$}
    \RightLabel{$\RAndL{\lev}{\lev}$}
    \UnaryInfC{$\GentzenS \proves^{h} A^\Vis \land B^\Vis, (\Gamma \setminus \{A \land B\})^\Vis, \Phi_1, \Box\Phi_2 \seq{\lev} \Psi, \Delta^\Vis$}
    \DisplayProof
  }
  The induction hypothesis applied to the premise yields $\GentzenSPL \proves A, B, \Gamma \setminus \{A \land B\}, \Phi_1, \Phi_2 \seq{\lev} \Psi, \Delta$,
  and $\RAndL{\lev}{\lev}$ of $\GentzenSPL$ yields the desired sequent.

  \medskip\noindent
  \textbf{Case} $\RAndR{\lev}{\lev}$:
  The principal formula is either $p^\Vis$ for some $p \in \Delta$, or $(A \land B)^\Vis$ for some $A \land B \in \Delta$.
  In the former case, the application we are concerned with is the following.
  \pagecenter{%
    \AxiomC{$\GentzenS \proves^{h-1} \Gamma^\Vis, \Phi_1, \Box\Phi_2 \seq{\lev} \Psi, p, (\Delta \setminus \{p\})^\Vis$}
    \AxiomC{$\GentzenS \proves^{h-1} \Gamma^\Vis, \Phi_1, \Box\Phi_2 \seq{\lev} \Psi, \Box p, (\Delta \setminus \{p\})^\Vis$}
    \RightLabel{$\RAndR{\lev}{\lev}$}
    \BinaryInfC{$\GentzenS \proves^{h} \Gamma^\Vis, \Phi_1, \Box\Phi_2 \seq{\lev} \Psi, p \land \Box p, (\Delta \setminus \{p\})^\Vis$}
    \DisplayProof
  }
  The induction hypothesis applied to the left premise yields
  $\GentzenSPL \proves \Gamma, \Phi_1, \Phi_2 \seq{\lev} \Psi, p, \Delta \setminus \{p\}$,
  whose succedent coincides with $\Psi, \Delta$ as a set, since $p \in \Delta$.
  Note that the right premise, whose succedent contains the boxed variable $\Box p$, is simply not used.

  In the latter case, the application we are concerned with is the following.
  \pagecenter{%
    \AxiomC{$\GentzenS \proves^{h-1} \Gamma^\Vis, \Phi_1, \Box\Phi_2 \seq{\lev} \Psi, A^\Vis, (\Delta \setminus \{A \land B\})^\Vis$}
    \AxiomC{$\GentzenS \proves^{h-1} \Gamma^\Vis, \Phi_1, \Box\Phi_2 \seq{\lev} \Psi, B^\Vis, (\Delta \setminus \{A \land B\})^\Vis$}
    \RightLabel{$\RAndR{\lev}{\lev}$}
    \BinaryInfC{$\GentzenS \proves^{h} \Gamma^\Vis, \Phi_1, \Box\Phi_2 \seq{\lev} \Psi, A^\Vis \land B^\Vis, (\Delta \setminus \{A \land B\})^\Vis$}
    \DisplayProof
  }
  The induction hypothesis applies to both premises, and $\RAndR{\lev}{\lev}$ of $\GentzenSPL$ yields the desired sequent.

  \medskip\noindent
  \textbf{Case} $\ROrL{\lev}{\lev}$, $\ROrR{\lev}{\lev}$:
  The principal formula can only be $(A \lor B)^\Vis = A^\Vis \lor B^\Vis$ for some $A \lor B \in \Gamma$ (resp.\ $A \lor B \in \Delta$).
  The induction hypothesis applies to the premises, and the corresponding rule of $\GentzenSPL$ yields the desired sequent.

  \medskip\noindent
  \textbf{Case} $\RBoxL{2}{2}$:
  The principal formula is either $\Box p$ for some $p \in \Phi_2$, or $(E \limp F)^\Vis = \Box(E^\Vis \limp F^\Vis)$ for some $E \limp F \in \Gamma$.

  In the former case, the application we are concerned with is the following.
  \pagecenter{%
    \AxiomC{$\GentzenS \proves^{h-1} p, \Gamma^\Vis, \Phi_1, \Box(\Phi_2 \setminus \{p\}) \seq{2} \Psi, \Delta^\Vis$}
    \RightLabel{$\RBoxL{2}{2}$}
    \UnaryInfC{$\GentzenS \proves^{h} \Box p, \Gamma^\Vis, \Phi_1, \Box(\Phi_2 \setminus \{p\}) \seq{2} \Psi, \Delta^\Vis$}
    \DisplayProof
  }
  The induction hypothesis applied to the premise, with the components $\Phi_1 \cup \{p\}$ and $\Phi_2 \setminus \{p\}$, yields
  $\GentzenSPL \proves \Gamma, \Phi_1, p, \Phi_2 \setminus \{p\} \seq{2} \Psi, \Delta$,
  whose antecedent coincides with $\Gamma, \Phi_1, \Phi_2$ as a set, since $p \in \Phi_2$.

  In the latter case, the application we are concerned with is the following.
  \pagecenter{%
    \AxiomC{$\GentzenS \proves^{h-1} E^\Vis \limp F^\Vis, (\Gamma \setminus \{E \limp F\})^\Vis, \Phi_1, \Box\Phi_2 \seq{2} \Psi, \Delta^\Vis$}
    \RightLabel{$\RBoxL{2}{2}$}
    \UnaryInfC{$\GentzenS \proves^{h} \Box(E^\Vis \limp F^\Vis), (\Gamma \setminus \{E \limp F\})^\Vis, \Phi_1, \Box\Phi_2 \seq{2} \Psi, \Delta^\Vis$}
    \DisplayProof
  }
  Applying the height-preserving inversion $\Inv{\RImpL{2}{2}}$ (Lemma \ref{lem:GentzenS_inversion_ImpL}) to the premise, we obtain the following.
  \[
    \begin{aligned}
      \GentzenS & \proves^{h-1} (\Gamma \setminus \{E \limp F\})^\Vis, \Phi_1, \Box\Phi_2 \seq{2} \Psi, \Delta^\Vis, E^\Vis  \\
      \GentzenS & \proves^{h-1} F^\Vis, (\Gamma \setminus \{E \limp F\})^\Vis, \Phi_1, \Box\Phi_2 \seq{2} \Psi, \Delta^\Vis.
    \end{aligned}
  \]
  Since the height is $h - 1$, the induction hypothesis applies and yields the following.
  \[
    \begin{aligned}
      \GentzenSPL & \proves (\Gamma \setminus \{E \limp F\}), \Phi_1, \Phi_2 \seq{2} \Psi, \Delta, E  \\
      \GentzenSPL & \proves F, (\Gamma \setminus \{E \limp F\}), \Phi_1, \Phi_2 \seq{2} \Psi, \Delta.
    \end{aligned}
  \]
  Using $\RImpL{2}{2}$ we obtain $\GentzenSPL \proves E \limp F, (\Gamma \setminus \{E \limp F\}), \Phi_1, \Phi_2 \seq{2} \Psi, \Delta$, that is, $\GentzenSPL \proves \Gamma, \Phi_1, \Phi_2 \seq{2} \Psi, \Delta$.

  \medskip\noindent
  \textbf{Case} $\RBoxGL{1}{1}$:
  Since the antecedent of the conclusion of an application of $\RBoxGL{1}{1}$ consists of boxed formulas only and its succedent consists of exactly one boxed formula,
  we have $\Phi_1 = \Psi = \emptyset$, $\Gamma = \{C_j \limp D_j : 0 \le j < k\}$ for some $k \geq 0$ and formulas $C_j, D_j$ ($0 \le j < k$), and $\Delta = \{E \limp F\}$ for some formulas $E, F$.
  The application we are concerned with is thus the following.
  \pagecenter{%
    \AxiomC{$\GentzenS \proves^{h-1} \{C^\Vis_j \limp D^\Vis_j : 0 \le j < k\}, \Phi_2, \Box\{C^\Vis_j \limp D^\Vis_j : 0 \le j < k\}, \Box\Phi_2, \Box(E^\Vis \limp F^\Vis) \seq{1} E^\Vis \limp F^\Vis$}
    \RightLabel{$\RBoxGL{1}{1}$}
    \UnaryInfC{$\GentzenS \proves^{h} \Box\{C^\Vis_j \limp D^\Vis_j : 0 \le j < k\}, \Box\Phi_2 \seq{1} \Box(E^\Vis \limp F^\Vis)$}
    \DisplayProof
  }
  Applying the height-preserving inversion $\Inv{\RImpR{1}{1}}$ (Lemma \ref{lem:GentzenS_inversion_ImpR}) to the premise, we obtain the following.
  \pagecenter{\ensuremath{%
      \GentzenS \proves^{h-1} \{C^\Vis_j \limp D^\Vis_j : 0 \le j < k\}, \Phi_2, \Box\{C^\Vis_j \limp D^\Vis_j : 0 \le j < k\}, \Box\Phi_2, \Box(E^\Vis \limp F^\Vis), E^\Vis \seq{1} F^\Vis.
    }}
  Furthermore, the height-preserving inversion $\Inv{\RImpL{1}{1}}$ (Lemma \ref{lem:GentzenS_inversion_ImpL}) applies to the members of $\{C^\Vis_j \limp D^\Vis_j : 0 \le j < k\}$, so, following the definition of $\RImpR{1}{1}$ in $\GentzenSPL$ and determining $\Gamma_i, \Delta_i$ from $C_0, \dots, C_{k-1}$ and $D_0, \dots, D_{k-1}$, for each $i$ with $0 \leq i < 2^k$ we obtain the following.
  \pagecenter{\ensuremath{%
      \GentzenS \proves^{h-1} \Delta^\Vis_i, \Phi_2, \Box\{C^\Vis_j \limp D^\Vis_j : 0 \le j < k\}, \Box\Phi_2, \Box(E^\Vis \limp F^\Vis), E^\Vis \seq{1} F^\Vis, \Gamma^\Vis_i.
    }}
  Each of these sequents is again of the form required in the statement:
  since $\Box\{C^\Vis_j \limp D^\Vis_j : 0 \le j < k\} = \{C_j \limp D_j : 0 \le j < k\}^\Vis$ and $\Box(E^\Vis \limp F^\Vis) = (E \limp F)^\Vis$,
  its antecedent consists of the translations of $\Delta_i, \{C_j \limp D_j : 0 \le j < k\}, E \limp F, E$, the variables $\Phi_2$, and the boxed variables $\Box\Phi_2$.
  Since the height is $h - 1$, the induction hypothesis applies and yields the following.
  \[
    \GentzenSPL \proves \Delta_i, \{C_j \limp D_j : 0 \le j < k\}, E \limp F, E, \Phi_2 \seq{1} F, \Gamma_i.
  \]
  Applying $\RImpR{1}{1}$ of $\GentzenSPL$ with $\Sigma := \{C_j \limp D_j : 0 \le j < k\} \cup \Phi_2$ to these premises, we obtain the following, where we note that antecedents are sets and hence the duplicated occurrences of $C_j \limp D_j$ are contracted.
  \[
    \GentzenSPL \proves \{C_j \limp D_j : 0 \le j < k\}, \Phi_2 \seq{1} E \limp F.
  \]
  This is the desired sequent $\Gamma, \Phi_1, \Phi_2 \seq{1} \Psi, \Delta$.
\end{proof}

\begin{lem}[{\cite[Lemma 5]{Pet23}}] \label{lem:embedding_S_to_SPL}
  If $\GentzenS \proves \Gamma^\Vis \seq{\lev} \Delta^\Vis$, then $\GentzenSPL \proves \Gamma \seq{\lev} \Delta$.
\end{lem}

\begin{proof}
  Take $\Phi_1 = \Phi_2 = \Psi = \emptyset$ in Lemma \ref{lem:embedding_S_to_SPL_general}.
\end{proof}

\begin{thm}[{\cite[Theorem 1]{Pet23}}] \label{thm:embedding_GentzenSPL_GentzenS}
  $\GentzenSPL \proves \Gamma \seq{\lev} \Delta$ iff $\GentzenS \proves \Gamma^\Vis \seq{\lev} \Delta^\Vis$.
\end{thm}

\begin{proof}
  This follows by combining Lemma \ref{lem:embedding_SPL_to_S} and Lemma \ref{lem:embedding_S_to_SPL}.
\end{proof}

Combining Theorem \ref{thm:embedding_GentzenSPL_GentzenS} with the properties of $\GentzenS$ (Propositions \ref{prop:GentzenS_provability} and \ref{prop:GentzenS_cut_elimination}), we immediately obtain the following corollaries.

\begin{cor} \label{cor:embedding_GentzenSPL_S}
  $\LogicSPL \proves A$ iff $\LogicS \proves A^\Vis$ for every $A \in \FmlPL$.
\end{cor}

\begin{cor}[Cut-admissibility of $\GentzenSPL$] \label{cor:GentzenSPL_cut_elimination}
  In $\GentzenSPL$, the cut rule $\RCut{\lev}{\lev}$ is admissible.
\end{cor}

\begin{rem} \label{rem:syntacticality_embedding}
  It should be noted that this embedding proof, and consequently Corollary \ref{cor:GentzenSPL_cut_elimination},
  cannot be regarded as purely syntactic in the strict sense.
  In Lemmas \ref{lem:embedding_SPL_to_S_aux1} and \ref{lem:embedding_SPL_to_S_aux2} we use the cut,
  relying on the fact that cut is admissible in $\GentzenS$.
  The sequent calculus for $\LogicS$ proposed by \cite{Kus20} comes with a purely syntactic cut-elimination algorithm;
  however, the cut-admissibility of $\GentzenS$ from \cite{KK23} that we adopt here
  (more precisely, the system modified by adding $\land$ and $\lor$ as primitive rules)
  is established by a semantic argument.
  Consequently, for instance, we cannot construct a cut-free proof-tree of $\GentzenSPL$ explicitly from this proof.
\end{rem}

\begin{rem}[About $\LogicFPL$ and $\LogicGL$]
  As also suggested by Petrukhin's paper,
  if one is given a suitable sequent calculus for $\LogicFPL$ (that is, roughly speaking, the fragment of $\GentzenSPL$ restricted to sequents $\seqI$)
  and a sequent calculus for $\LogicGL$ (likewise, the fragment of $\GentzenS$ restricted to sequents $\seqI$, which coincides with \cite{SV82}),
  then, by reconstructing the argument above while forgetting about $\seqII$,
  one obtains a syntactic construction of the embedding of $\LogicFPL$ into $\LogicGL$ established in \cite{Vis81}.
\end{rem}

\subsection{Errors in Petrukhin's proof} \label{sect:errors_in_Petrukhin}

This embedding result is based on Petrukhin \cite{Pet23}, but the proof in the original paper appears to contain several errors.
We point them out below.
The issue concerns \cite[Lemma 5]{Pet23}, which corresponds to Lemma \ref{lem:embedding_S_to_SPL} in the present paper.

\begin{prop*}[{\cite[Lemma 5]{Pet23}}]
  If $\GentzenS \proves \Gamma^\Vis, \Phi, \Box\Phi \seq{\lev} \Psi, \Delta^\Vis$, then $\GentzenSPL \proves \Gamma, \Phi \seq{\lev} \Psi, \Delta$.
  Here $\Gamma, \Delta$ are multisets of formulas and $\Phi, \Psi$ are multisets of propositional variables.
\end{prop*}

The proof proceeds by induction on the height of the proof-tree of $\GentzenS \proves \Gamma^\Vis, \Phi, \Box\Phi \seq{\lev} \Psi, \Delta^\Vis$.
Petrukhin claims that the proof proceeds as in \cite[Lemma 7]{YS17}, but the problem occurs in the case of $\RBoxL{2}{2}$.
Petrukhin verifies only the case where the principal formula is $\Box(E^\Vis \limp F^\Vis)$.
To begin with, this verification itself is erroneous.
The proof-tree given there is the following (cf.\ \cite[p.~17]{Pet23}).
\pagecenter{%
  \AxiomC{$\GentzenS \proves^{h-1} E^\Vis \limp F^\Vis, \Gamma^\Vis, \Phi \seq{2} \Psi, \Delta^\Vis$}
  \RightLabel{$\RBoxL{2}{2}$}
  \UnaryInfC{$\GentzenS \proves^{h} \Box(E^\Vis \limp F^\Vis), \Gamma^\Vis, \Phi \seq{2} \Psi, \Delta^\Vis$}
  \DisplayProof
}
There are two problems.
First, $\Box\Phi$ is missing from both the conclusion and the premise.
Second, even if we supply it and regard the premise as $\GentzenS \proves^{h-1} E^\Vis \limp F^\Vis, \Gamma^\Vis, \Phi, \Box\Phi \seq{2} \Psi, \Delta^\Vis$, the induction hypothesis cannot be applied to this premise.
This is because $E^\Vis \limp F^\Vis$ never belongs to $\Sigma^\Vis$ for any $\Sigma$, and is of course not a propositional variable, so the premise is not of the form $\Gamma^\Vis, \Phi, \Box\Phi \seq{\lev} \Psi, \Delta^\Vis$.
Nevertheless, Petrukhin appears to apply the induction hypothesis to this premise and obtain $E \limp F, \Gamma, \Phi \seq{2} \Psi, \Delta$, which is not justified.

Furthermore, one should in fact also consider the case where, for some $p \in \Phi$, the principal formula is $\Box p$.
In this case the proof-tree is as follows.
\pagecenter{%
  \AxiomC{$\GentzenS \proves^{h-1} p, \Gamma^\Vis, p, (\Phi \setminus \{p\}), \Box(\Phi \setminus \{p\}) \seq{2} \Psi, \Delta^\Vis$}
  \RightLabel{$\RBoxL{2}{2}$}
  \UnaryInfC{$\GentzenS \proves^{h} \Box p, \Gamma^\Vis, p, \Phi \setminus \{p\}, \Box(\Phi \setminus \{p\}) \seq{2} \Psi, \Delta^\Vis$}
  \DisplayProof
}
Now, in order to apply the induction hypothesis to this premise and carry the argument through, one would need to obtain, by \emph{height-preserving} weakening and contraction on the premise, $\GentzenS \proves^{h-1} p, \Box p, \Gamma^\Vis, (\Phi \setminus \{p\}), \Box(\Phi \setminus \{p\}) \seq{2} \Psi, \Delta^\Vis$.
However, such height-preserving weakening and contraction cannot be carried out in $\GentzenS$\footnote{The calculus adopted by \cite{Pet23} is the sequent calculus of \cite{Kus20}, but this does not affect the details of the argument.} (cf.\ \cite{NvP11}).
The reason this problem did not arise in the proof of \cite{YS17} is that the sequent system for $\LogicKFour$ adopted there is of G3-type (cf.\ \cite{TS00}).
In that system, the weakening, contraction, and inversion rules are height-preservingly admissible\footnote{However, \cite{YS17} does not make explicit how these properties are established for their extension to $\LogicKFour$.}.

In this paper, these problems are avoided in Lemma \ref{lem:embedding_S_to_SPL_general} as follows.
Since our sequents are based on \emph{sets}, contraction is built into the formulation and never needs to be applied explicitly, let alone height-preservingly.
Moreover, the problematic case where the principal formula of $\RBoxL{2}{2}$ is $\Box p$ is absorbed by allowing the two sets of variables $\Phi_1$ and $\Phi_2$ in the statement to differ:
unboxing $\Box p$ merely moves $p$ from $\Phi_2$ to $\Phi_1$, so the induction hypothesis applies directly, without any use of weakening.

\section{Sequent calculus $\GentzenD$ for modal logic $\LogicD$} \label{sect:GentzenD}

Kashima et al.\ \cite{KKIM25} introduced two sequent calculi for $\LogicD$: one using only two-level sequents and another using three-level sequents
\footnote{Cut-elimination does not hold for their two-level calculus, so it is not suitable for our purpose (cf.\ \cite[Theorem 4.9]{KKIM25}).}.
Here we present the system $\GentzenD$, obtained by modifying their three-level calculus in the same way as in Section \ref{sect:GentzenS}, together with its basic proof-theoretic properties.

\begin{defn}
  Let $\Gamma, \Delta$ be sets of formulas and let $\lev = 1, 2, 3$.
  An expression of the form $\Gamma \seq{\lev} \Delta$ is called a \emph{sequent}.
  The sequent calculus $\GentzenD$ for $\LogicD$, in which sequents of level $\lev = 1, 2, 3$ occur, is defined by the following rules.

  \begin{center}
    \begin{minipage}{0.46\textwidth}\centering
      \begin{prooftree}
        \AxiomC{}
        \RightLabel{$\RAx{\lev}$}
        \UnaryInfC{$p \seq{\lev} p$}
      \end{prooftree}
      where $p \in \PropVar$.
    \end{minipage}\hfill
    \begin{minipage}{0.46\textwidth}\centering
      \begin{prooftree}
        \AxiomC{}
        \RightLabel{$\RBotL{\lev}$}
        \UnaryInfC{$\bot \seq{\lev} {}$}
      \end{prooftree}
    \end{minipage}

    \medskip

    \begin{minipage}{0.46\textwidth}\centering
      \begin{prooftree}
        \AxiomC{$\Gamma \seq{\lev} \Delta$}
        \RightLabel{$\RWL{\lev}{\lev}$}
        \UnaryInfC{$A, \Gamma \seq{\lev} \Delta$}
      \end{prooftree}
    \end{minipage}\hfill
    \begin{minipage}{0.46\textwidth}\centering
      \begin{prooftree}
        \AxiomC{$\Gamma \seq{\lev} \Delta$}
        \RightLabel{$\RWR{\lev}{\lev}$}
        \UnaryInfC{$\Gamma \seq{\lev} \Delta, A$}
      \end{prooftree}
    \end{minipage}

    \medskip

    \begin{minipage}{0.46\textwidth}\centering
      \begin{prooftree}
        \AxiomC{$A, B, \Gamma \seq{\lev} \Delta$}
        \RightLabel{$\RAndL{\lev}{\lev}$}
        \UnaryInfC{$A \land B, \Gamma \seq{\lev} \Delta$}
      \end{prooftree}
    \end{minipage}\hfill
    \begin{minipage}{0.46\textwidth}\centering
      \begin{prooftree}
        \AxiomC{$\Gamma \seq{\lev} \Delta, A$}
        \AxiomC{$\Gamma \seq{\lev} \Delta, B$}
        \RightLabel{$\RAndR{\lev}{\lev}$}
        \BinaryInfC{$\Gamma \seq{\lev} \Delta, A \land B$}
      \end{prooftree}
    \end{minipage}

    \medskip

    \begin{minipage}{0.46\textwidth}\centering
      \begin{prooftree}
        \AxiomC{$A, \Gamma \seq{\lev} \Delta$}
        \AxiomC{$B, \Gamma \seq{\lev} \Delta$}
        \RightLabel{$\ROrL{\lev}{\lev}$}
        \BinaryInfC{$A \lor B, \Gamma \seq{\lev} \Delta$}
      \end{prooftree}
    \end{minipage}\hfill
    \begin{minipage}{0.46\textwidth}\centering
      \begin{prooftree}
        \AxiomC{$\Gamma \seq{\lev} \Delta, A, B$}
        \RightLabel{$\ROrR{\lev}{\lev}$}
        \UnaryInfC{$\Gamma \seq{\lev} \Delta, A \lor B$}
      \end{prooftree}
    \end{minipage}

    \medskip

    \begin{minipage}{0.46\textwidth}\centering
      \begin{prooftree}
        \AxiomC{$\Gamma \seq{\lev} \Delta, A$}
        \AxiomC{$B, \Gamma \seq{\lev} \Delta$}
        \RightLabel{$\RImpL{\lev}{\lev}$}
        \BinaryInfC{$A \limp B, \Gamma \seq{\lev} \Delta$}
      \end{prooftree}
    \end{minipage}\hfill
    \begin{minipage}{0.46\textwidth}\centering
      \begin{prooftree}
        \AxiomC{$A, \Gamma \seq{\lev} \Delta, B$}
        \RightLabel{$\RImpR{\lev}{\lev}$}
        \UnaryInfC{$\Gamma \seq{\lev} \Delta, A \limp B$}
      \end{prooftree}
    \end{minipage}

    \medskip

    \begin{minipage}{0.46\textwidth}\centering
      \begin{prooftree}
        \AxiomC{$\Gamma \seq{1} \Delta$}
        \RightLabel{$\RLift{1}{2}$}
        \UnaryInfC{$\Gamma \seq{2} \Delta$}
      \end{prooftree}
    \end{minipage}\hfill
    \begin{minipage}{0.46\textwidth}\centering
      \begin{prooftree}
        \AxiomC{$\Box\Gamma \seq{2} \Box\Delta$}
        \RightLabel{$\RLift{2}{3}$}
        \UnaryInfC{$\Box\Gamma \seq{3} \Box\Delta$}
      \end{prooftree}
    \end{minipage}

    \medskip

    \begin{minipage}{0.46\textwidth}\centering
      \begin{prooftree}
        \AxiomC{$\Box\Gamma, \Gamma, \Box A \seq{1} A$}
        \RightLabel{$\RBoxGL{1}{1}$}
        \UnaryInfC{$\Box\Gamma \seq{1} \Box A$}
      \end{prooftree}
    \end{minipage}\hfill
    \begin{minipage}{0.46\textwidth}\centering
      \begin{prooftree}
        \AxiomC{$A, \Gamma \seq{2} \Delta$}
        \RightLabel{$\RBoxL{2}{2}$}
        \UnaryInfC{$\Box A, \Gamma \seq{2} \Delta$}
      \end{prooftree}
    \end{minipage}
  \end{center}
\end{defn}

The provability of $\seq{1}$ and $\seq{2}$ coincides with that of $\GentzenS$, and the level-$3$ sequents $\seqIII$ capture provability in $\LogicD$.

\begin{rem}
  Our calculus differs from the three-level system of \cite{KKIM25} in the same two respects as in Remark \ref{rem:GentzenS_differences},
  and the facts stated below are likewise recovered semantically.
\end{rem}

In particular, the identity sequents for arbitrary formulas are derivable.

\begin{lem} \label{lem:GentzenD_axioms}
  For every $A \in \FmlML$, $\GentzenD \proves A \seq{\lev} A$.
\end{lem}

\begin{proof}
  The argument is the same as for Lemma \ref{lem:GentzenS_original_axioms}.
\end{proof}

\begin{prop}[Provability of $\GentzenD$ {\cite[Proposition 3.6]{KKIM25}}] \label{prop:GentzenD_provability}
  The following equivalences hold.
  \begin{enumerate}
    \item $\GentzenD \proves \Gamma \seq{1} \Delta$ iff $\LogicGL \proves \lconj \Gamma \limp \ldisj \Delta$.
    \item $\GentzenD \proves \Gamma \seq{2} \Delta$ iff $\LogicS \proves \lconj \Gamma \limp \ldisj \Delta$.
    \item $\GentzenD \proves \Gamma \seq{3} \Delta$ iff $\LogicD \proves \lconj \Gamma \limp \ldisj \Delta$.
  \end{enumerate}
\end{prop}

\begin{prop}[Cut-admissibility of $\GentzenD$ {\cite[Theorem 4.8]{KKIM25}}] \label{prop:GentzenD_cut_elimination}
  The cut rule is admissible in $\GentzenD$.
  That is, with the cut rule $\RCut{\lev}{\lev}$ defined as below, if $\GentzenD + \RCut{\lev}{\lev} \proves \Gamma \seq{\lev} \Delta$ then $\GentzenD \proves \Gamma \seq{\lev} \Delta$.
  \begin{prooftree}
    \AxiomC{$\Gamma_1 \seq{\lev} \Delta_1, A$}
    \AxiomC{$A, \Gamma_2 \seq{\lev} \Delta_2$}
    \RightLabel{$\RCut{\lev}{\lev}$}
    \BinaryInfC{$\Gamma_1, \Gamma_2 \seq{\lev} \Delta_1, \Delta_2$}
  \end{prooftree}
\end{prop}

\begin{lem}[Inversion rules $\Inv{\RImpL{\lev}{\lev}}$, $\Inv{\RImpR{\lev}{\lev}}$] \label{lem:GentzenD_inversion}
  In $\GentzenD$ the inversion rules $\Inv{\RImpL{\lev}{\lev}}$ and $\Inv{\RImpR{\lev}{\lev}}$ are height-preservingly admissible.
  \begin{center}
    \begin{minipage}{0.46\textwidth}\centering
      \begin{prooftree}
        \AxiomC{$\GentzenD \proves^h A \limp B, \Gamma \seq{\lev} \Delta$}
        \RightLabel{$\Inv{\RImpL{\lev}{\lev}}$}
        \UnaryInfC{$\GentzenD \proves^h \Gamma \seq{\lev} \Delta, A$}
      \end{prooftree}
    \end{minipage}\hfill
    \begin{minipage}{0.46\textwidth}\centering
      \begin{prooftree}
        \AxiomC{$\GentzenD \proves^h A \limp B, \Gamma \seq{\lev} \Delta$}
        \RightLabel{$\Inv{\RImpL{\lev}{\lev}}$}
        \UnaryInfC{$\GentzenD \proves^h B, \Gamma \seq{\lev} \Delta$}
      \end{prooftree}
    \end{minipage}

    \medskip

    \begin{prooftree}
      \AxiomC{$\GentzenD \proves^h \Gamma \seq{\lev} \Delta, A \limp B$}
      \RightLabel{$\Inv{\RImpR{\lev}{\lev}}$}
      \UnaryInfC{$\GentzenD \proves^h A, \Gamma \seq{\lev} \Delta, B$}
    \end{prooftree}
  \end{center}
\end{lem}

\begin{proof}
  The argument is the same as for Lemmas \ref{lem:GentzenS_inversion_ImpL} and \ref{lem:GentzenS_inversion_ImpR}.
\end{proof}

\section{Sequent calculus $\GentzenDPL$ for propositional logic $\LogicDPL$} \label{sect:GentzenDPL}

We define $\GentzenDPL$, the sequent calculus for $\LogicDPL$, which is obtained by adding to $\GentzenSPL$ a single rule $\RLift{2}{3}$ that arises naturally from the lift rule of $\GentzenD$.

\begin{defn}
  Let $\Gamma, \Delta$ be sets of formulas and let $\lev = 1, 2, 3$.
  An expression of the form $\Gamma \seq{\lev} \Delta$ is called a \emph{sequent}.
  The sequent calculus $\GentzenDPL$ for $\LogicDPL$, in which sequents of level $\lev = 1, 2, 3$ occur, is defined by the following rules.

  \begin{center}
    \begin{minipage}{0.46\textwidth}\centering
      \begin{prooftree}
        \AxiomC{}
        \RightLabel{$\RAx{\lev}$}
        \UnaryInfC{$p \seq{\lev} p$}
      \end{prooftree}
      where $p \in \PropVar$.
    \end{minipage}\hfill
    \begin{minipage}{0.46\textwidth}\centering
      \begin{prooftree}
        \AxiomC{}
        \RightLabel{$\RBotL{\lev}$}
        \UnaryInfC{$\bot \seq{\lev} {}$}
      \end{prooftree}
    \end{minipage}

    \medskip

    \begin{minipage}{0.46\textwidth}\centering
      \begin{prooftree}
        \AxiomC{$\Gamma \seq{\lev} \Delta$}
        \RightLabel{$\RWL{\lev}{\lev}$}
        \UnaryInfC{$A, \Gamma \seq{\lev} \Delta$}
      \end{prooftree}
    \end{minipage}\hfill
    \begin{minipage}{0.46\textwidth}\centering
      \begin{prooftree}
        \AxiomC{$\Gamma \seq{\lev} \Delta$}
        \RightLabel{$\RWR{\lev}{\lev}$}
        \UnaryInfC{$\Gamma \seq{\lev} \Delta, A$}
      \end{prooftree}
    \end{minipage}

    \medskip

    \begin{minipage}{0.46\textwidth}\centering
      \begin{prooftree}
        \AxiomC{$A, B, \Gamma \seq{\lev} \Delta$}
        \RightLabel{$\RAndL{\lev}{\lev}$}
        \UnaryInfC{$A \land B, \Gamma \seq{\lev} \Delta$}
      \end{prooftree}
    \end{minipage}\hfill
    \begin{minipage}{0.46\textwidth}\centering
      \begin{prooftree}
        \AxiomC{$\Gamma \seq{\lev} \Delta, A$}
        \AxiomC{$\Gamma \seq{\lev} \Delta, B$}
        \RightLabel{$\RAndR{\lev}{\lev}$}
        \BinaryInfC{$\Gamma \seq{\lev} \Delta, A \land B$}
      \end{prooftree}
    \end{minipage}

    \medskip

    \begin{minipage}{0.46\textwidth}\centering
      \begin{prooftree}
        \AxiomC{$A, \Gamma \seq{\lev} \Delta$}
        \AxiomC{$B, \Gamma \seq{\lev} \Delta$}
        \RightLabel{$\ROrL{\lev}{\lev}$}
        \BinaryInfC{$A \lor B, \Gamma \seq{\lev} \Delta$}
      \end{prooftree}
    \end{minipage}\hfill
    \begin{minipage}{0.46\textwidth}\centering
      \begin{prooftree}
        \AxiomC{$\Gamma \seq{\lev} \Delta, A, B$}
        \RightLabel{$\ROrR{\lev}{\lev}$}
        \UnaryInfC{$\Gamma \seq{\lev} \Delta, A \lor B$}
      \end{prooftree}
    \end{minipage}

    \medskip

    \begin{minipage}{0.46\textwidth}\centering
      \begin{prooftree}
        \AxiomC{$\Gamma \seq{2} \Delta, A$}
        \AxiomC{$B, \Gamma \seq{2} \Delta$}
        \RightLabel{$\RImpL{2}{2}$}
        \BinaryInfC{$A \limp B, \Gamma \seq{2} \Delta$}
      \end{prooftree}
    \end{minipage}\hfill
    \begin{minipage}{0.46\textwidth}\centering
      \begin{prooftree}
        \AxiomC{$\Gamma \seq{1} \Delta$}
        \RightLabel{$\RLift{1}{2}$}
        \UnaryInfC{$\Gamma \seq{2} \Delta$}
      \end{prooftree}
    \end{minipage}

    \medskip

    \begin{prooftree}
      \AxiomC{$\Delta_i, \Sigma, A \limp B, A \seq{1} B, \Gamma_i : 0 \leq i < 2^k$}
      \RightLabel{$\RImpR{1}{1}$}
      \UnaryInfC{$\Sigma, \{C_j \limp D_j : 0 \le j < k\} \seq{1} A \limp B$}
    \end{prooftree}

    where $k$, $\Delta_i$ and $\Gamma_i$ are defined in the same way as in $\GentzenSPL$.

    \medskip

    \begin{prooftree}
      \AxiomC{$\{A_i \limp B_i : 0 \le i < n\}, \Phi \seq{2} \{C_j \limp D_j : 0 \le j < m\}$}
      \RightLabel{$\RLift{2}{3}$}
      \UnaryInfC{$\{A_i \limp B_i : 0 \le i < n\}, \Phi \seq{3} \{C_j \limp D_j : 0 \le j < m\}$}
    \end{prooftree}

    where $n, m \geq 0$, the formulas $A_0, \dots, A_{n-1}, B_0, \dots, B_{n-1}, C_0, \dots, C_{m-1}, D_0, \dots, D_{m-1}$ are arbitrary, and $\Phi$ is a set of propositional variables.
  \end{center}
\end{defn}

The reason why propositional variables are allowed to occur in the antecedent of $\RLift{2}{3}$, but not in the succedent, will be explained in Section \ref{sect:Embedding_GentzenDPL_GentzenD} (Lemmas \ref{lem:GD_box_to_Vis} and \ref{lem:GD_Vis_to_box}, and Remark \ref{rem:Lift_no_succedent_variables}).

As in Section \ref{sect:GentzenSPL}, the identity sequent is derivable for every formula.

\begin{lem} \label{lem:GentzenDPL_axioms}
  For every $A \in \FmlPL$, $\GentzenDPL \proves A \seq{\lev} A$.
\end{lem}

\begin{proof}
  The argument is the same as for Lemma \ref{lem:GentzenSPL_axioms}.
  When $\lev = 3$ and $A \equiv B \limp C$, we further apply $\RLift{2}{3}$ to $B \limp C \seq{2} B \limp C$, which is applicable since both sides consist of implications.
\end{proof}

We define \emph{Dzhaparidze Propositional Logic} $\LogicDPL$ as the set of formulas $\{A \in \FmlPL : \GentzenDPL \proves \seqIII A\}$.

\section{Embedding of $\GentzenDPL$ into $\GentzenD$} \label{sect:Embedding_GentzenDPL_GentzenD}

The embedding of $\GentzenDPL$ into $\GentzenD$ is obtained by adapting the argument of Section \ref{sect:Embedding_GentzenSPL_GentzenS},
where the only new rule is $\RLift{2}{3}$.

\begin{lem} \label{lem:GentzenD_admits_BoxK4}
  The rule $\RBoxKFour{1}{\lev}$ is admissible in $\GentzenD$.
  \begin{prooftree}
    \AxiomC{$\Box\Gamma, \Gamma \seq{1} A$}
    \RightLabel{$\RBoxKFour{1}{\lev}$}
    \UnaryInfC{$\Box\Gamma \seq{\lev} \Box A$}
  \end{prooftree}
\end{lem}

\begin{proof}
  As in Lemma \ref{lem:GentzenS_admits_BoxK4}, we first obtain $\Box\Gamma \seq{1} \Box A$.
  When $\lev = 2$, we apply $\RLift{1}{2}$, and when $\lev = 3$, we apply $\RLift{1}{2}$ and $\RLift{2}{3}$, the latter being applicable since both $\Box\Gamma$ and $\Box A$ consist of boxed formulas.
\end{proof}

\begin{lem} \label{lem:embedding_DPL_to_D_aux1}
  For every $A \in \FmlPL$, $\GentzenD \proves A^\Vis \seq{\lev} \Box A^\Vis$.
\end{lem}

\begin{proof}
  Same as for Lemma \ref{lem:embedding_SPL_to_S_aux1}.
\end{proof}

\begin{lem} \label{lem:embedding_DPL_to_D_aux2}
  The following rule is admissible in $\GentzenD$.
  \begin{prooftree}
    \AxiomC{$\Sigma^\Vis, \Gamma, \Box A \seq{1} A$}
    \UnaryInfC{$\Sigma^\Vis, \Box\Gamma \seq{\lev} \Box A$}
  \end{prooftree}
\end{lem}

\begin{proof}
  As in Lemma \ref{lem:embedding_SPL_to_S_aux2}, we first obtain $\Box\Sigma^\Vis, \Box\Gamma \seq{1} \Box A$ by $\RWL{1}{1}$ and $\RBoxGL{1}{1}$.
  Since every formula in this sequent is boxed, we can lift it to $\Box\Sigma^\Vis, \Box\Gamma \seq{\lev} \Box A$ by applying $\RLift{1}{2}$ and $\RLift{2}{3}$ as appropriate.
  Then, by Lemma \ref{lem:embedding_DPL_to_D_aux1} and $\RCut{\lev}{\lev}$, we successively replace the elements of $\Box\Sigma^\Vis$ and obtain $\Sigma^\Vis, \Box\Gamma \seq{\lev} \Box A$.
\end{proof}

The following two lemmas explain the shape of the rule $\RLift{2}{3}$ of $\GentzenDPL$:
in $\GentzenD$, the formulas $p^\Vis = p \land \Box p$ and $\Box p$ are interderivable at level $2$, while at level $3$ only the direction from $p^\Vis$ to $\Box p$ survives.
Hence a boxed variable $\Box p$ occurring in the antecedent of an application of $\RLift{2}{3}$ of $\GentzenD$ corresponds, on the propositional side, to the bare propositional variable $p$.

\begin{lem} \label{lem:GD_box_to_Vis}
  For every $p \in \PropVar$, $\GentzenD \proves \Box p \seq{2} p^\Vis$.
\end{lem}

\begin{proof}
  Follows from:
  \begin{prooftree}
    \AxiomC{}
    \RightLabel{$\RAx{2}$}
    \UnaryInfC{$p \seq{2} p$}
    \RightLabel{$\RBoxL{2}{2}$}
    \UnaryInfC{$\Box p \seq{2} p$}
    \AxiomC{\scriptsize(Lemma \ref{lem:GentzenD_axioms})}
    \noLine
    \UnaryInfC{$\Box p \seq{2} \Box p$}
    \RightLabel{$\RAndR{2}{2}$}
    \BinaryInfC{$\Box p \seq{2} p \land \Box p$}
  \end{prooftree}
\end{proof}

\begin{lem} \label{lem:GD_Vis_to_box}
  For every $p \in \PropVar$, $\GentzenD \proves p^\Vis \seq{3} \Box p$.
\end{lem}

\begin{proof}
  Follows from:
  \begin{prooftree}
    \AxiomC{\scriptsize(Lemma \ref{lem:GentzenD_axioms})}
    \noLine
    \UnaryInfC{$\Box p \seq{3} \Box p$}
    \RightLabel{$\RWL{3}{3}$}
    \UnaryInfC{$p, \Box p \seq{3} \Box p$}
    \RightLabel{$\RAndL{3}{3}$}
    \UnaryInfC{$p \land \Box p \seq{3} \Box p$}
  \end{prooftree}
\end{proof}

\begin{rem} \label{rem:Lift_no_succedent_variables}
  We do not allow propositional variables to occur in the succedent of the rule $\RLift{2}{3}$ of $\GentzenDPL$.
  As we will see in the proof of Lemma \ref{lem:embedding_DPL_to_D}, such an extension would require $\GentzenD \proves \Box p \seq{3} p^\Vis$, and hence $\GentzenD \proves \Box p \seq{3} p$, which is not available since the reflection $\Box A \limp A$ is not a theorem of $\LogicD$.
  This asymmetry does not cause any problem in the converse direction either: the succedent of the conclusion of an application of $\RLift{2}{3}$ of $\GentzenD$ consists of boxed formulas, and the boxed formulas among the translations are exactly those of implications, so no propositional variable arises there.
\end{rem}

\begin{lem} \label{lem:embedding_DPL_to_D}
  If $\GentzenDPL \proves \Gamma \seq{\lev} \Delta$, then $\GentzenD \proves \Gamma^\Vis \seq{\lev} \Delta^\Vis$.
\end{lem}

\begin{proof}
  We show by induction on the height of the proof-tree.
  We consider the case of $\RLift{2}{3}$; the remaining cases are handled exactly as in Lemma \ref{lem:embedding_SPL_to_S}, using Lemmas \ref{lem:embedding_DPL_to_D_aux1} and \ref{lem:embedding_DPL_to_D_aux2} in place of Lemmas \ref{lem:embedding_SPL_to_S_aux1} and \ref{lem:embedding_SPL_to_S_aux2}.

  \medskip\noindent
  \textbf{Case} $\RLift{2}{3}$:
  Suppose the rule we now wish to consider has the following form, where $\Phi$ is a set of propositional variables.
  \begin{prooftree}
    \AxiomC{$\GentzenDPL \proves^{h-1} \{A_i \limp B_i : 0 \le i < n\}, \Phi \seq{2} \{C_j \limp D_j : 0 \le j < m\}$}
    \RightLabel{$\RLift{2}{3}$}
    \UnaryInfC{$\GentzenDPL \proves^{h} \{A_i \limp B_i : 0 \le i < n\}, \Phi \seq{3} \{C_j \limp D_j : 0 \le j < m\}$}
  \end{prooftree}
  By the induction hypothesis, we have $\GentzenD \proves \{\Box(A^\Vis_i \limp B^\Vis_i) : 0 \le i < n\}, \Phi^\Vis \seq{2} \{\Box(C^\Vis_j \limp D^\Vis_j) : 0 \le j < m\}$.
  Here every formula on both sides other than $\Phi^\Vis$ is boxed, whereas the formulas $p^\Vis = p \land \Box p$ for $p \in \Phi$ are not, so the rule $\RLift{2}{3}$ of $\GentzenD$ is not applicable yet.
  First, using Lemma \ref{lem:GD_box_to_Vis} and $\RCut{2}{2}$ with the cut formula $p^\Vis$, we successively replace each $p^\Vis$ in the antecedent by $\Box p$, and obtain
  \[
    \GentzenD \proves \{\Box(A^\Vis_i \limp B^\Vis_i) : 0 \le i < n\}, \Box\Phi \seq{2} \{\Box(C^\Vis_j \limp D^\Vis_j) : 0 \le j < m\}.
  \]
  Since every formula in this sequent is boxed, the rule $\RLift{2}{3}$ of $\GentzenD$ applies and yields
  \[
    \GentzenD \proves \{\Box(A^\Vis_i \limp B^\Vis_i) : 0 \le i < n\}, \Box\Phi \seq{3} \{\Box(C^\Vis_j \limp D^\Vis_j) : 0 \le j < m\}.
  \]
  Finally, using Lemma \ref{lem:GD_Vis_to_box} and $\RCut{3}{3}$ with the cut formula $\Box p$, we successively replace each $\Box p$ back by $p^\Vis$, and obtain
  \[
    \GentzenD \proves \{\Box(A^\Vis_i \limp B^\Vis_i) : 0 \le i < n\}, \Phi^\Vis \seq{3} \{\Box(C^\Vis_j \limp D^\Vis_j) : 0 \le j < m\},
  \]
  which is nothing but $\GentzenD \proves \{A_i \limp B_i : 0 \le i < n\}^\Vis, \Phi^\Vis \seq{3} \{C_j \limp D_j : 0 \le j < m\}^\Vis$.
\end{proof}

As in Section \ref{sect:Embedding_GentzenSPL_GentzenS}, the converse direction is proved in the following strengthened form.

\begin{lem} \label{lem:embedding_D_to_DPL_general}
  Let $\Gamma, \Delta$ be sets of propositional formulas and let $\Phi_1, \Phi_2, \Psi$ be sets of propositional variables.
  If $\GentzenD \proves \Gamma^\Vis, \Phi_1, \Box\Phi_2 \seq{\lev} \Psi, \Delta^\Vis$,
  then $\GentzenDPL \proves \Gamma, \Phi_1, \Phi_2 \seq{\lev} \Psi, \Delta$.
\end{lem}

\begin{proof}
  As in Lemma \ref{lem:embedding_S_to_SPL_general}, no formula in the assumed sequent is an implication, and hence the last rule of a proof-tree of the assumed sequent is neither $\RImpL{\lev}{\lev}$ nor $\RImpR{\lev}{\lev}$; note that this also excludes these rules at level $3$.
  We show by induction on the height of the proof-tree.
  All the cases other than $\RLift{2}{3}$ are handled exactly as in Lemma \ref{lem:embedding_S_to_SPL_general}.

  \medskip\noindent
  \textbf{Case} $\RLift{2}{3}$:
  We have $\Phi_1 = \Psi = \emptyset$ since every formula in the conclusion of an application of $\RLift{2}{3}$ of $\GentzenD$ is boxed.
  Moreover, since the boxed formulas among the translations are exactly those of implications,
  we have $\Gamma = \{A_i \limp B_i : 0 \le i < n\}$ and $\Delta = \{C_j \limp D_j : 0 \le j < m\}$ for some $n, m \geq 0$ and formulas $A_i, B_i$ ($0 \le i < n$), $C_j, D_j$ ($0 \le j < m$).
  Thus, the application we are concerned with is the following.
  \begin{prooftree}
    \AxiomC{$\GentzenD \proves^{h-1} \{A_i \limp B_i : 0 \le i < n\}^\Vis, \Box\Phi_2 \seq{2} \{C_j \limp D_j : 0 \le j < m\}^\Vis$}
    \RightLabel{$\RLift{2}{3}$}
    \UnaryInfC{$\GentzenD \proves^{h} \{A_i \limp B_i : 0 \le i < n\}^\Vis, \Box\Phi_2 \seq{3} \{C_j \limp D_j : 0 \le j < m\}^\Vis$}
  \end{prooftree}
  The induction hypothesis applied to the premise yields $\GentzenDPL \proves \{A_i \limp B_i : 0 \le i < n\}, \Phi_2 \seq{2} \{C_j \limp D_j : 0 \le j < m\}$.
  Since $\{A_i \limp B_i : 0 \le i < n\}$ and $\{C_j \limp D_j : 0 \le j < m\}$ consist of implications and $\Phi_2$ consists of propositional variables,
  the rule $\RLift{2}{3}$ of $\GentzenDPL$ applies and yields
  $\GentzenDPL \proves \{A_i \limp B_i : 0 \le i < n\}, \Phi_2 \seq{3} \{C_j \limp D_j : 0 \le j < m\}$,
  which is the desired sequent since $\Phi_1 = \emptyset$.
\end{proof}

\begin{lem} \label{lem:embedding_D_to_DPL}
  If $\GentzenD \proves \Gamma^\Vis \seq{\lev} \Delta^\Vis$, then $\GentzenDPL \proves \Gamma \seq{\lev} \Delta$.
\end{lem}

\begin{proof}
  Take $\Phi_1 = \Phi_2 = \Psi = \emptyset$ in Lemma \ref{lem:embedding_D_to_DPL_general}.
\end{proof}

\begin{thm} \label{thm:embedding_GentzenDPL_GentzenD}
  $\GentzenDPL \proves \Gamma \seq{\lev} \Delta$ iff $\GentzenD \proves \Gamma^\Vis \seq{\lev} \Delta^\Vis$.
\end{thm}

\begin{proof}
  This follows by combining Lemma \ref{lem:embedding_DPL_to_D} and Lemma \ref{lem:embedding_D_to_DPL}.
\end{proof}

Combining Theorem \ref{thm:embedding_GentzenDPL_GentzenD} with the properties of $\GentzenD$ (Propositions \ref{prop:GentzenD_provability} and \ref{prop:GentzenD_cut_elimination}), we immediately obtain the following corollaries.

\begin{cor} \label{cor:embedding_GentzenDPL_D}
  $\LogicDPL \proves A$ iff $\LogicD \proves A^\Vis$ for every $A \in \FmlPL$.
\end{cor}

\begin{cor}[Cut-admissibility of $\GentzenDPL$] \label{cor:GentzenDPL_cut_elimination}
  In $\GentzenDPL$, the cut rule $\RCut{\lev}{\lev}$ is admissible.
\end{cor}

We note that the discussion on whether the embedding is purely syntactic is similar to that in Remark \ref{rem:syntacticality_embedding}.

\section{Concluding Remarks and Open Problems} \label{sect:conclusion}

In this paper, we examined the proof by Petrukhin \cite{Pet23} of the embedding of the propositional logic $\LogicSPL$ into Solovay's provability logic $\LogicS$, and gave a correct proof.
Furthermore, extending that proof, we proposed a new propositional logic $\LogicDPL$ by means of the sequent calculus for Japaridze's provability logic $\LogicD$ introduced by Kashima et al.\ \cite{KKIM25}, and proved the corresponding embedding as well.
In the rest of this section, we discuss some issues concerning $\LogicSPL$ and $\LogicDPL$ that were not addressed in this paper, together with directions for future work.

\subsection{Semantics}

We did not discuss the properties that the logics $\LogicSPL$ and $\LogicDPL$ have in their own right.
For instance, $\LogicFPL$ as a provability logic was already considered in Visser's original paper \cite{Vis81}.
Since $\LogicSPL$ and $\LogicDPL$ are the counterparts of $\LogicS$ and $\LogicD$ respectively,
it is natural to expect that this line of investigation makes sense for them as well.
Moreover, Petrukhin \cite{Pet23} also examines natural deduction and semantics for $\LogicSPL$.
In particular, as for semantics, he introduces a suitable counterpart of Visser's tail models \cite{Vis84}
on the propositional side and proves its completeness through the embedding.
From the perspective of Kripke semantics for propositional logics,
Visser's translation maps $p$ to $p \land \Box p$ precisely so that $\Box p$ restores the persistency%
\footnote{For a Kripke model $\langle W, R, \forces \rangle$, if $x \forces p$ and $x R y$ then $y \forces p$.}
of Kripke models.
On the other hand, we believe that it should also be possible to define semantics for $\LogicSPL$ directly
and to prove its completeness without going through the embedding, but so far we have not succeeded.
As for $\LogicD$, semantics has been defined by Beklemishev \cite{Bek89,Bek90} and by Kashima et al.\ \cite{KKIM25}.
Note that a discussion of the semantics of $\LogicDPL$ via the embedding would be possible, but we did not present it here.

We consider the problem of semantics to be rather important, for the following somewhat subtle reason.
Suppose that there is a propositional logic that corresponds (or \emph{should} correspond) to $\LogicS$ or $\LogicD$.
Then, for practical purposes, we consider that the justification of the correspondence should be witnessed by
some semantics, such as Kripke-style or arithmetical interpretations.
In this paper, we defined the logics $\LogicSPL$ and $\LogicDPL$ by sequent calculi,
and we cannot eliminate the possibility that they are, in a sense, ad hoc systems built solely in order to claim the embeddability.
In other words, it is conceivable that different logics $\LogicSPL'$ and $\LogicDPL'$,
obtained by slightly modifying the definitions of $\GentzenSPL$ and $\GentzenDPL$, can also be embedded into $\LogicS$ and $\LogicD$;
in such a case, one may ask what the difference between them really is, or object that they are \emph{merely} syntactically different.
If, on the other hand, the correspondence is matched precisely at the semantic level,
one could respond to such criticism with some substantial grounds.
For these reasons, we regard the semantic witnessing of $\LogicSPL$ and $\LogicDPL$ as a major problem for future work.

\subsection{Hilbert-style formulations}

Hilbert-style systems for $\LogicBPL$ and $\LogicFPL$ have been studied in \cite{SO97,IKK01}.
A natural question is then what Hilbert-style systems for $\LogicSPL$ and $\LogicDPL$ look like.
However, we defined $\LogicSPL$ and $\LogicDPL$ by sequent calculi, and moreover by sequent calculi with several levels of sequents.
Defining Hilbert-style systems seems to be far from easy.
In fact, we conjecture that no Hilbert-style system exists in the language with a single level of implication $\limp$.
If this conjecture is true, further questions arise:
what is a logic in a language equipped with several levels of implications $\limp^1$ and $\limp^2$,
and what do such implications mean?

\subsection{Syntactic cut-elimination}

As we noted in Remark \ref{rem:syntacticality_embedding},
the embeddability and the cut-admissibility of our systems for $\LogicSPL$ and $\LogicDPL$ rely
on the cut-admissibility of the sequent calculi for $\LogicS$ and $\LogicD$,
and the latter is established by semantic arguments, so our results are not purely syntactic.
It is still open to give a constructive cut-elimination algorithm for the sequent calculi $\GentzenSPL$ and $\GentzenDPL$.

\section*{Acknowledgements}

The author would like to thank Kohei Tominaga, Naoyuki Hatanaka, Haruka Kogure, and Taishi Kurahashi
for carefully reading the first draft of this paper
and providing valuable discussions and corrections of minor errors.

\subsection*{Use of AI}

The author declares that AI/LLMs were involved in the preparation of this paper.
While writing this draft, a review by Claude Fable pointed out a fatal error
in a part of the proof in Section \ref{sect:Embedding_GentzenSPL_GentzenS} as originally given.
This part was fixed with Claude, and the resulting arguments were verified by the author.
Any remaining errors are the sole responsibility of the author.

\bibliographystyle{plain}
\bibliography{refs}

\end{document}